\documentclass[final,12pt]{amsart}
    \title[Mutations of $w$-simple-minded systems in $\sT_w$]{Mutations of simple-minded systems in Calabi-Yau categories generated by a spherical object}
    \date{\today}

\author{Raquel Coelho Sim\~oes}
\address{Centro de An\'alise Funcional, Estruturas Lineares e Aplica\c{c}\~oes,Faculdade de ci\^encias da Universidade de Lisboa, Campo Grande, Edif\'icio C6, Piso 2, 1749-016 Lisboa, Portugal.}
\email{rcoelhosimoes@campus.ul.pt}

\usepackage{amssymb}
\usepackage{latexsym}
\usepackage{amsmath}
\usepackage{mathrsfs}
\usepackage[all]{xy}
\usepackage{enumerate}	
\usepackage{enumitem}
\usepackage{paralist}
\usepackage{graphicx}

\usepackage{verbatim}

\usepackage[usenames,dvipsnames]{xcolor}
\usepackage{hyperref}
\hypersetup{
    colorlinks,
    linkcolor={red!50!black},
    citecolor={blue!50!black},
    urlcolor={blue!80!black}
}

\newcommand{\hyref}[2]{ \hyperref[#2]{#1~\ref*{#2}} }

\usepackage[notref, notcite]{showkeys} 

\usepackage{color}

\theoremstyle{plain}
\newtheorem{theorem}{Theorem}[section]

\newtheorem{lemma}[theorem]{Lemma}
\newtheorem{corollary}[theorem]{Corollary}
\newtheorem{proposition}[theorem]{Proposition}

\newtheorem{introtheorem}{Theorem}

\theoremstyle{definition}
\newtheorem{remark}[theorem]{Remark}
\newtheorem{example}[theorem]{Example}
\newtheorem*{definition}{Definition} 
\newtheorem*{definitions}{Definitions} 

\newtheorem*{notation}{Notation}

\newcommand{\Case}[1]{\medskip\noindent\emph{Case #1:}}

\DeclareMathAlphabet{\mathpzc}{OT1}{pzc}{m}{it}

\newcommand{\cP}{\mathscr{P}}

\newcommand{\sA}{\mathsf{A}}

\newcommand{\sC}{\mathsf{C}}

\newcommand{\sK}{\mathsf{K}}

\newcommand{\sS}{\mathsf{S}}
\newcommand{\sT}{\mathsf{T}}

\newcommand{\sX}{\mathsf{X}}
\newcommand{\sY}{\mathsf{Y}}

\newcommand{\bN}{\mathbb{N}}

\newcommand{\bZ}{\mathbb{Z}}

\newcommand{\tA}{\mathtt{A}}

\newcommand{\tS}{\mathtt{S}}

\newcommand{\tU}{\mathtt{U}}

\newcommand{\tX}{\mathtt{X}}

\newcommand{\Tw}{\mathsf{T}_w}
\newcommand{\Ainf}{A_{\infty}}

\renewcommand{\geq}{\geqslant}
\renewcommand{\leq}{\leqs}
\renewcommand{\phi}{\varphi}
\renewcommand{\epsilon}{\varepsilon}

\newcommand{\cocone}[1]{\mathsf{cocone}(#1)}

\DeclareMathOperator{\Hom}{\mathsf{Hom}}

\DeclareMathOperator{\Ext}{\mathsf{Ext}}

\newcommand{\ind}[1]{\mathsf{ind}(#1)}

\DeclareMathOperator{\modulo}{\mathrm{mod}}

\newcommand{\kk}{{\mathbf{k}}}

\newcommand{\SSS}{\mathbb{S}}

\newcommand{\thick}[2]{\mathsf{thick}_{#1}(#2)}

\newcommand{\extn}[1]{\langle #1 \rangle}

\newcommand{\leqs}{\leqslant}

\newcommand{\tri}[3]{#1\rightarrow #2\rightarrow #3\rightarrow \Sigma #1}

\newcommand{\trilabels}[6]{#1\stackrel{#4}{\longrightarrow} #2\stackrel{#5}{\longrightarrow} #3\stackrel{#6}{\longrightarrow} \Sigma #1}

\newcommand{\rayto}[1]{\mathsf{ray}_{\! -}(#1)}

\newcommand{\exray}[2]{\mathsf{exray}_{#1}(#2)}
\newcommand{\homfrom}[1]{\mathsf{H}^+(#1)}
\newcommand{\homto}[1]{\mathsf{H}^-(#1)}
\newcommand{\arc}[1]{\mathtt{#1}}

\entrymodifiers={+!!<0pt,\fontdimen22\textfont2>}

\newcommand{\rowmat}[2]{\big[#1 \ #2\big]}
\newcommand{\arr}{\ar[ur] \ar[dr]}
\newcommand{\arup}{\ar@/^/[u]}
\newcommand{\ardn}{\ar@/^/[d]}

\newcommand{\coloneqq}{\mathrel{\mathop:}=}

\newcommand{\smxy}[1]{{\text{\tiny$#1$}}}

\usepackage{tikz}
	\usetikzlibrary{decorations.pathmorphing}
        \usetikzlibrary{topaths}
        \usetikzlibrary{automata,arrows}
        \usetikzlibrary{shapes.geometric} 

\begin{document}

\begin{abstract}
In this article, we give a definition and a classification of `higher' simple-minded systems in triangulated categories generated by spherical objects with negative Calabi-Yau dimension. We also study mutations of this class of objects and that of `higher' Hom-configurations and Riedtmann configurations. This gives an explicit analogue of the `nice' mutation theory exhibited in cluster-tilting theory.
\end{abstract}



\maketitle

\addtocontents{toc}{\protect{\setcounter{tocdepth}{-1}}}  
\section*{Introduction} 
\addtocontents{toc}{\protect{\setcounter{tocdepth}{1}}}   

Simple-minded systems \cite{KL} and Riedtmann configurations \cite{Riedtmann} (also known as maximal systems of stable orthogonal bricks \cite{Po}) appear in the context of selfinjective algebras in the study of a well-known conjecture of Auslander and Reiten. In that setup, simple-minded systems and Riedtmann configurations are $(-1)$-Calabi-Yau (CY).

In this paper we introduce {\it $w$-simple-minded systems}, for a negative integer $w$, generalising simple-minded systems(= $(-1)$-simple-minded systems) in the sense that a $w$-simple-minded system is $w$-CY. The Calabi-Yau property is an important duality condition introduced by Kontsevich, which has since appeared in many areas in mathematics and physics. The main motivational example for us comes from cluster-tilting theory in representation theory, an important recent breakthrough which generalises classical tilting theory and gives a categorification of cluster algebras. 

Initially, cluster-tilting theory took place in the cluster category, an orbit category of the bounded derived category of a finite-dimensional hereditary algebra. The theory was then developed to higher analogues of these categories, so called $m$-cluster categories, for $m \geq 2$, and led to the systematic study of positive CY triangulated categories, which reveal many interesting and important homological and combinatorial properties. Not much is known, however, about negative CY triangulated categories, although there is recent progress in \cite{CS1, CS2, CS3, CSP, HJY}.

We study $w$-simple-minded objects in triangulated categories $\Tw$ generated by a $w$-spherical object. These categories, which can be defined for any integer $w$ and form an important class of $w$-CY triangulated categories, have been subject of intensive study (see \cite{FY, HJ1, HJY, J, KYZ, Ng, ZZ}), especially in connection with cluster-tilting theory, for $w \geq 2$.
 
The main object of study in $w$-cluster categories, and other positive CY triangulated categories, are $w$-cluster tilting objects. One remarkable property of these objects, which was the feature that enabled the first categorification of cluster algebras in \cite{BMRRT}, is \textit{mutation}, an operation that allows us to construct a new object from a given one by replacing an indecomposable summand. Connections between mutations of cluster-tilting objects and mutations of other notions such as simple-minded collections and t-structures \cite{King-Qiu,KY} and silting objects and exceptional collections \cite{Aihara-Iyama, BRT} have now been widely explored. In the case of finite-dimensional algebras, mutation can be thought of as a discrete version of the wall-crossing phenomena that occurs on the Bridgeland stability manifold \cite{Bridgeland}. Therefore, studying the mutation theory of various Calabi-Yau configurations is an important problem. 

In \cite{HJ2}, a combinatorial classification of $(w-1)$-cluster tilting and weakly $(w-1)$-cluster tilting subcategories in $\Tw$ ($w \geq 2$) is given and used to study their mutation theory. In \cite{CS3}, a geometric model of $\Tw$ in terms of {\it admissible arcs} of an $\infty$-gon is used to give a classification of $w$-Riedtmann configurations and $w$-Hom configurations, a larger class of objects which includes $w$-Riedtmann configurations. In this paper we shall see that $w$-simple-minded systems are $w$-Riedtmann configurations and, using the computation of middle terms of extensions in $\Tw$ studied in \cite{CSP}, we shall prove the following theorem.

\begin{introtheorem}[see Theorem \ref{thm:classificationsms} and Corollary \ref{cor:functoriallyfinite}]\label{introthm:sms}
Let $w \leq -1$, $\sS$ be a $w$-Riedtmann configuration in $\Tw$ and $\tS$ the corresponding set of admissible arcs. The following are equivalent:
\begin{compactenum}
\item $\sS$ is a $w$-simple-minded system;
\item Every arc in $\tS$ has an overarc in $\tS$;
\item The extension closure of $\sS$ is functorially finite. 
\end{compactenum}
\end{introtheorem}  

It is interesting to note that $(2)$ and $(3)$ are also a feature of $(w-1)$-cluster tilting subcategories in the positive CY case; cf. \cite[Lemma 3.4]{HJ2}. It was not known thus far whether the notions of maximal systems of stable orthogonal bricks and simple-minded systems coincided (cf. \cite{Dugas12}). As a byproduct of our combinatorial classification, we can conclude that these notions do not coincide.  

We shall then use the classification of $w$-Hom configurations, $w$-Riedtmann configurations and $w$-simple-minded systems to study their mutation theory, establishing an analogue to the mutation theory described in \cite{HJ2}. The first similarity is that one is able to replace one indecomposable by another indecomposable object. This mutation theory is described more precisely in the following theorems. 

\begin{introtheorem} [see Corollary \ref{cor:replacementssms}]
Let $w \leq -1$, $\sS$ be a $w$-simple-minded system or a $w$-Riedtmann configuration with $|w|-1$ outer-isolated vertices in $\Tw$, and $s \in \sS$ be indecomposable. Then $\sS$ can be mutated at $s$ in precisely $|w|-1$ ways.   
\end{introtheorem}

\begin{introtheorem} [see Corollary \ref{cor:replacementsl}]
Let $w \leq -1$ and $0 \leq l \leq |w|-2$. Then there exists a $w$-Riedtmann configuration $\sS_l$ of $\Tw$ with an indecomposable object $s$ such that $\sS_l$ can be mutated at $s$ in precisely $l$ ways. 
\end{introtheorem}

\begin{introtheorem} [see Corollary \ref{cor:replacementsriedtmann}]
Let $w \leq -1$, $\sS$ be a $w$-Riedtmann configuration of $\Tw$ and $s \in \sS$ be indecomposable. Then $\sS$ can be mutated at $s$ in at most $|w|-1$ ways. 
\end{introtheorem}

Comparing these theorems with those in \cite{HJ2}, we notice that, for $w \leq -2$, $w$-simple-minded systems in $\Tw$ have exactly the same mutation behaviour as that of $(|w|-1)$-cluster tilting subcategories of $\sT_{|w|}$, and $w$-Riedtmann configurations in $\Tw$ have exactly the same mutation behaviour as that of weakly $(|w|-1)$-cluster tilting subcategories of $\sT_{|w|}$. There is, however, another class of objects in $\Tw$ which is disjoint to the class of $w$-simple-minded systems and yet has the same `nice' mutation behaviour, in the sense that every indecomposable object has a fixed number of replacements given in terms of the CY dimension. However, there is a fundamental difference between these two classes of objects, namely functorial finiteness of their extension closure, as stated in Theorem \ref{introthm:sms}. This is also the difference between weakly $(|w|-1)$-cluster tilting subcategories and $(|w|-1)$-cluster tilting subcategories. Hence, it is our belief that $w$-simple-minded systems should be seen as the natural analogue of $(|w|-1)$-cluster tilting subcategories. 

We finish the study of mutations by realising the mutation behaviour of $w$-simple-minded systems in terms of approximation theory, using the functorial finiteness property of their extension closure. We shall see that the mutation theory defined for simple-minded systems (cf. \cite{Dugas12}, see also \cite{KY}) can be applied to $w$-simple-minded systems in an iterative way, in order to obtain the $|w|-1$ replacements of a given indecomposable object. Again, we point out the similarities with the mutation theory of $(|w|-1)$-cluster tilting objects in $(|w|-1)$-cluster categories \cite{Zhu08}.  

\section{Definitions and basic properties}

In this section we give the definition of the categories and classes of objects studied in this paper. We also collect the main known results on these categories and objects, which will be used later. Given a category $\sC$, we shall denote the collection of (isomorphism classes of) indecomposable objects by $\ind{\sC}$. We always assume that every subcategory of $\sT$ is full and contains the zero object.

\subsection{The categories}
Let $\kk$ be an algebraically closed field. Let $\sT$ be a $\kk$-linear triangulated category and $w \in \bZ$.  A {\it $w$-spherical object} $s$ of $\sT$ is an object satisfying the following conditions:
\[
\Hom_{\sT}(s, \Sigma^i s) = 
\left\{
\begin{array}{ll}
\kk          & \text{if } i=0,w; \\
0            & \text{otherwise, and}
\end{array}
\right.
\]
\[
\Hom_{\sT}(s,t) \simeq D\Hom_{\sT}(t,\Sigma^w s), \text{ for all $t \in \sT$}.
\]
Here, $D(-)\coloneqq \Hom_{\kk}(-,\kk)$ denotes the usual vector space duality. 

An object $s$ \emph{generates} $\sT$ if $\thick{\sT}{s} = \sT$, i.e. the smallest triangulated subcategory containing $s$ that is also closed under direct summands is $\sT$.

Let $\Tw$ be a $\kk$-linear triangulated category that is idempotent complete and generated by a $w$-spherical object. It is known that $\Tw$ is unique up to triangle equivalence \cite[Theorem 2.1]{KYZ}. The category $\Tw$ is Hom-finite and Krull-Schmidt. Moreover, $\Tw$ is  $w$-CY, i.e. $\SSS \simeq \Sigma^w$, where $\SSS$ denotes the Serre functor. In other words, we have a bifunctorial isomorphism:
\[
\Hom_{\Tw} (x,y) \simeq D\Hom_{\Tw} (y,\Sigma^wx),
\]
for every object $x, y \in \Tw$.
 
It was seen in \cite[Proposition 1.10]{HJY} that the AR-quiver of $\Tw$ consists of $|w-1|$ copies of $\bZ A_\infty$. Given one copy, the others are obtained by applying $\Sigma, \ldots, \Sigma^{|w-1|-1}$.

\begin{proposition}[{\cite[Propositions 3.2 and 3.3]{HJY}}] \label{prop:hom-hammocks}
Let $x,y \in \ind{\Tw}$ for $w\neq 1$. Then
\[
\dim_{\kk} \Hom_{\Tw}(x,y) = 
\left\{
\begin{array}{ll}
1 & \text{if } y \in \homfrom{x} \cup \homto{\SSS x}, \\
0 & \text{otherwise.}
\end{array}
\right.
\]
\end{proposition}

We refer the reader to~\cite[Subsection 2.2.]{CSP} for the notions of {\it forward hammock} $\homfrom{x}$ and {\it backward hammock} $\homto{x}$ of an object $x$. 

\subsection{The objects}
Let $\sT$ be a $\kk$-linear triangulated category, with Serre functor $\SSS$, and let $w \leq 0$.  

A subcategory $\sX$ is {\it extension-closed} if, for every triangle $\tri{a}{b}{c}$ with $a, c \in \sX$, we have $b \in \sX$.

Given subcategories $\sX, \sY$ of $\sT$, we set up the following notation: 
\begin{compactitem}
\item $\sX * \sY \coloneqq \{t \in \sT \mid \exists \ \tri{x}{t}{y} \text{ with } x \in \sX \text{ and } y \in \sY\}$.
\item $\sX^\oplus$ is the smallest full subcategory of $\sT$ containing $\sX$ which is closed under summands.
\item $(\sX)_1 \coloneqq \sX$ and $(\sX)_n \coloneqq  (\sX * (\sX)_{n-1})^\oplus$, for $n \geq 2$.
\item $\extn{\sX}$ is the smallest extension-closed full subcategory of $\sT$ containing $\sX$. 
\item Given $i \in \bZ$, $\sX^{\perp_i} \coloneqq \{ y \in \sT \mid \Ext^i (x,y) = 0, \forall x \in \sX\}$. Similarly, $\,^{\perp_i} \sX \coloneqq \{ y \in \sT \mid \Ext^i (y,x) = 0, \forall x \in \sX\}$.
\end{compactitem}

\begin{definition}[{\cite{CS3}}]\label{def:m-Hom-config}
A collection $\sS$ of indecomposable objects in $\sT$ is called a \textit{$w$-Hom configuration} if it satisfies the following conditions:
\begin{compactitem}
\item $s \in \sS \Leftrightarrow \begin{cases} 
      s \in (\sS)^{\perp_i}, \text{ for } i = w+1, \ldots, -1 & \text{and } \\
      s \in (\sS \setminus \{s\})^{\perp_i} , \text{ for } i = w, 0.
\end{cases}$
\item $s \in \sS \Leftrightarrow \begin{cases} 
s \in \,^{\perp_i} (\sS), \text{ for } i = w+1, \ldots, -1 & \text{and } \\
s \in \,^{\perp_i} (\sS \setminus \{s\}), \text{ for } i = w, 0.
\end{cases}$
\end{compactitem}
\end{definition}

If $\sT$ is $w$-CY then the two conditions in the definition above are equivalent. 

\begin{definitions} Let $\sS = \{s_i \mid i \in I \}$ be a set of indecomposable objects of $\sT$. 
\begin{compactenum}
\item $\sS$ is \textit{$w$-orthogonal} if it satisfies the following conditions:
 \begin{compactenum}
  \item[(a)] $\Hom (s_i, s_j) = \delta_{ij} \kk$,
  \item[(b)] If $w \leq -2$, $\Ext^k (s_i,s_j) = 0$, for $w+1 \leq k \leq -1$ and $i, j \in I$.
 \end{compactenum}
\item $\sS$ is a \textit{left (right, resp.) $w$-Riedtmann configuration} if it is a $w$-orthogonal set satisfying:
 \begin{compactenum}
  \item [(c)] For every $x \in \sT$, there are $i \in I$ and $w+1 \leq k \leq 0$ such that $\Ext^k (s_i,x) \neq 0$ ($\Ext^k (x,s_i) \neq 0$, resp.).
 \end{compactenum}
\item $\sS$ is a \textit{$w$-Riedtmann configuration} if it is both a left and right $w$-Riedtmann configuration. 
\item $\sS$ is a \textit{$w$-simple-minded system} if it is a $w$-orthogonal set satisfying:
 \begin{compactenum}
  \item [(d)] $\bigcup\limits_{n \geq 0} (\sX)_n^\oplus = \sT, \text{ where } \sX \coloneqq \bigcup\limits_{i=w+1}^0 \Sigma^i \sS$.
 \end{compactenum}
\end{compactenum}
\end{definitions} 

For brevity we shall sometimes refer to $(-1)$-orthogonal sets as simply orthogonal sets. 

\begin{lemma}[{\cite[Lemmas 2.2, 2.3 and 2.7]{Dugas12}}]\label{lem:closedsummands}
Let $\sX$ be a subcategory of $\sT$ and $n \geq 1$.
\begin{compactenum}
\item $(\sX)_n = ((\sX)_{n-1} * \sX)^\oplus$, and $\extn{\sX}^\oplus = \bigcup\limits_{n \geq 1} (\sX)_n$.
\item If $\sX$ is orthogonal, then $(\sX)_n = \sX * (\sX)_{n-1} = (\sX)_{n-1} * \sX$, for every $n \geq 1$. In particular, $\extn{\sX}^\oplus = \extn{\sX}$.  
\end{compactenum}
\end{lemma}

When $w = -1$, we recover the notions of a system of orthogonal bricks (=$(-1)$-orthogonal set) and a maximal system of orthogonal bricks (=$(-1)$-Riedtmann configuration). We also recover the notion of simple-minded system when $w=-1$, by Lemma \ref{lem:closedsummands}. Moreover, condition (d) means that the smallest subcategory of $\sT$ containing $\bigcup\limits_{i=w+1}^{0} \Sigma^i \sS$ which is closed under extensions and summands is $\sT$ itself. 

In the context of stable module categories of selfinjective categories, an additional requirement that a Riedtmann configuration or a simple-minded system $\sS$ is stable under $\SSS \Sigma$ is imposed. The natural analogue for a $w$-Riedtmann configuration or $w$-simple-minded system would be to impose stability under $\SSS \Sigma^{-w}$. However, since we are considering these objects in $w$-CY categories, this condition is redundant. 

\subsection{Combinatorial model} Here we recall the natural extension of the combinatorial model for $\Tw$ from \cite{HJ2} in the case $w\geq 2$ to the case $w \leq -1$. This will be in terms of `arcs/diagonals of the $\infty$-gon': each pair of integers $(t,u)$ is viewed as an arc connecting the integers $t$ and $u$ arranged on a number line.

Set $d \coloneqq w-1$. A pair of integers $(t,u)$ is called a \emph{$d$-admissible arc} if one has $u-t \leq w$ and $u-t \equiv 1 \modulo d$.

The \emph{length} of the arc $\arc{a} \coloneqq (t,u)$ is $|u-t|$, and it is denoted by $l (\arc{a})$. If $\arc{a}=(t,u)$ then the \emph{starting point} is $s(\arc{a}):=t$ and the \emph{ending point} or \emph{target} is $t(\arc{a}):=u$.

When $d$ is clear from context we refer to $d$-admissible arcs simply as \emph{admissible arcs}. We shall identify (isoclasses of) indecomposable objects in $\Tw$ with the corresponding admissible arcs, and we shall freely switch between objects and arcs. We will use Roman typeface for the indecomposable objects of $\Tw$ and typewriter typeface for the corresponding arcs. 

Figure~\ref{fig:AR-quiver2} shows how admissible arcs correspond to the indecomposable objects of $\Tw$. The action of the suspension, AR translate and Serre functor are given by
\[
\Sigma(t,u) = (t-1,u-1) \quad
\tau(t,u) = (t-d,u-d) \quad
\SSS(t,u) = (t-w,u-w).
\]

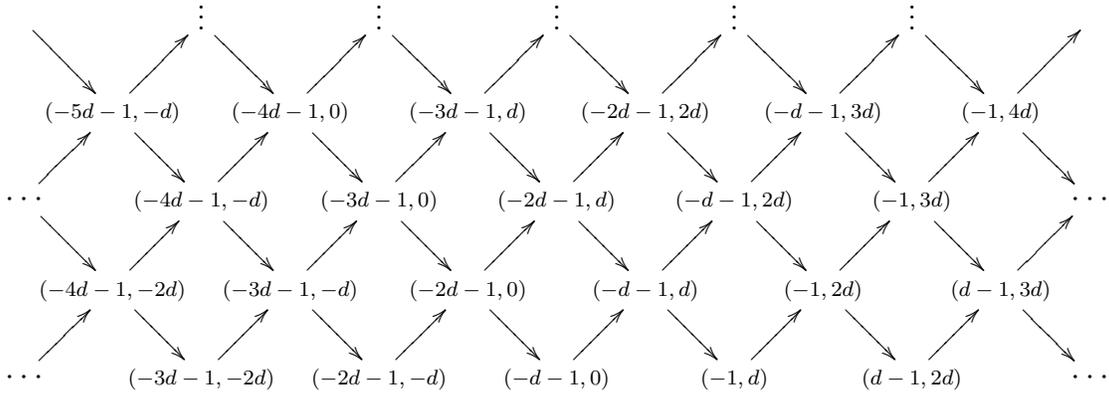
\begin{figure}[!ht]
\centerline{
\xymatrix@!=0.8pc{
\ar[dr]    &                      & \vdots \ar[dr]           &                   & \vdots \ar[dr]                  &          & \vdots \ar[dr] &                      & \vdots \ar[dr]       &                        & \vdots \ar[dr]         &                      &  \\
                        & \smxy{(-5d-1,-d)} \arr &                        & \smxy{(-4d-1,0)} \arr &                      & \smxy{(-3d-1,d)} \arr &                 & \smxy{(-2d-1,2d)} \arr &                        & \smxy{(-d-1,3d)} \arr  &                         & \smxy{(-1,4d)} \arr &                  \\
\cdots \arr        &                      & \smxy{(-4d-1,-d)}  \arr         &                   & \smxy{(-3d-1,0)}  \arr            &          & \smxy{(-2d-1,d)} \arr &                      & \smxy{(-d-1,2d)}  \arr     &                        & \smxy{(-1,3d)}  \arr    &                      & \cdots \\
                        & \smxy{(-4d-1,-2d)} \arr &                        & \smxy{(-3d-1,-d)} \arr &                      & \smxy{(-2d-1,0)} \arr &                 & \smxy{(-d-1,d)} \arr &                        & \smxy{(-1,2d)} \arr  &                         & \smxy{(d-1,3d)} \arr &                  \\
\cdots \ar[ur] &                      & \smxy{(-3d-1,-2d)} \ar[ur] &                   & \smxy{(-2d-1,-d)} \ar[ur] &          & \smxy{(-d-1,0)} \ar[ur]     &                     & \smxy{(-1,d)} \ar[ur] &                       & \smxy{(d-1,2d)} \ar[ur] &                      & \cdots
}}
\caption{\label{fig:AR-quiver2} A component of type $\bZ \Ainf$ with the endpoints of the $d$-admissible arcs described.}
\end{figure} 

\begin{definitions}
Let $\tX$ be a collection of admissible arcs in $\Tw$, $\arc{a} = (t,u), \arc{b} = (v,w) \in \tX$, and $x$ a vertex in the $\infty$-gon. 
\begin{compactenum}
\item The arc $\arc{b}$ is an \textit{overarc of $\arc{a}$} if $t(\arc{b}) < t(\arc{a}) < s(\arc{a}) < s(\arc{b})$. The arc $\arc{a}$ is called an \textit{innerarc of $\arc{b}$}. 
\item The arc $\arc{b}$ is an \textit{outer-arc} if it does not have any overarc in $\tX$.
\item The arcs $\arc{a}$ and $\arc{b}$ are \textit{crossing arcs} if they are incident with the same vertex or they {\it strictly cross}, i.e. either $u < w < t < v$ or $w < u < v < t$. 
\item The arc $\arc{b}$ is an \textit{overarc of $x$} if $t(\arc{b}) < x < s(\arc{b})$.
\item The vertex $x$ is an {\it inner-vertex} if there is an overarc of $x$ in $\tX$. If $\arc{b}$ is its smallest overarc, then $x$ is an {\it inner-vertex of $\arc{b}$}. 
\item The vertex $x$ is an \textit{isolated vertex} if it is not incident with any arc in $\tX$.
\end{compactenum}
We make the obvious definitions of inner-isolated vertex and inner-isolated vertex of an arc $\arc{b}$. If $v$ is isolated but not an inner-vertex, it will be called an {\it outer-isolated vertex}. 
\end{definitions}

\subsection{Some known classification results} 
For $w \leq -1$, $w$-Hom and $w$-Riedtmann configurations have been classified in \cite{CS3} in terms of admissible arcs.

\begin{theorem}[{\cite{CS3}\label{thm:Hom-config}}]\label{Homconfigclassification}
A set of admissible arcs $\tS$ is a $w$-Hom configuration in $\Tw$ if and only if the following conditions hold:
\begin{compactenum}
\item There are no crossings.
\item Any arc in $\tS$ has precisely $|w|-1$ inner-isolated vertices.
\item There are at most $|w|$ outer-isolated vertices. 
\end{compactenum}
\end{theorem}

Note that any (left/right) $|w|$-Riedtmann configuration is a $|w|$-Hom configuration.

\begin{theorem}[{\cite{CS3}}]\label{thm:Riedtmann}
Let $\sS$ be $w$-Hom configuration in $\Tw$ and $\tS$ the corresponding collection of admissible arcs. The following assertions are equivalent:
\begin{compactenum}
\item $\sS$ is a left $w$-Riedtmann configuration.
\item $\sS$ is a right $w$-Riedtmann configuration.
\item There are at most $|w|-1$ outer-isolated vertices in $\tS$.
\end{compactenum}
\end{theorem}

\subsection{Extensions in $\Tw$}
In \cite{CSP} a characterisation of extension-closed subcategories of $\Tw$ in terms of Ptolemy arcs was given.

\begin{definitions}
Let $\arc{a} = (t, u), \arc{b} = (v, w)$ be two admissible arcs of $\Tw$. 
\begin{compactenum}
\item If $\arc{a}$ and $\arc{b}$ are strictly crossing arcs, then the \textit{Ptolemy arcs of class I} associated to $\arc{a}$ and $\arc{b}$ are the remaining four arcs connecting the vertices incident with $\arc{a}$ and $\arc{b}$. 
\item Two Ptolemy arcs of class I (associated to the same pair of strictly crossing arcs) will be called \textit{opposite to each other} if they are not incident with a common vertex.
\item The \textit{distance} between $\arc{a}$ and $\arc{b}$ is defined as
\[
d(\arc{a},\arc{b}) \coloneqq \min\{|t - v|,|u - w|,|t - w|,|u-v|\}.
\]
\item The arcs $\arc{a}$ and $\arc{b}$ are \textit{neighbouring} if they do not cross and $d(\arc{a},\arc{b}) = 1$. 
\item If $\arc{a}$ and $\arc{b}$ are neighbouring arcs and $d(\arc{a},\arc{b})$ is given by the distance between vertices $x$ and $x-1$, then the corresponding \textit{Ptolemy arc of class II} is the arc connecting the vertices incident with $\arc{a}$ and $\arc{b}$ which are not $x$ and $x-1$.   
 \end{compactenum}
Figures \ref{fig:ptolemy-arcs} and \ref{fig:modified-ptolemy} illustrate these concepts. We shall sometimes refer to Ptolemy arcs of both classes simply as Ptolemy arcs.
\end{definitions}

\begin{figure}[!ht]
\begin{center}
\begin{tikzpicture}[thick,scale=0.75, every node/.style={scale=0.75},out=90,in=90]
\draw (0,1) -- (12,1);

\path (1,1) edge (7,1);
\path (5,1) edge (11,1);

\path[out=45,in=135,dashed] (1,1) edge (5,1);
\path[dashed] (1,1) edge (11,1);
\path[out=45,in=135,dashed] (5,1) edge (7,1);
\path[out=45,in=135,dashed] (7,1) edge (11,1);

\end{tikzpicture}
\end{center}
\caption{The Ptolemy arcs of class I.}
\label{fig:ptolemy-arcs}
\end{figure}
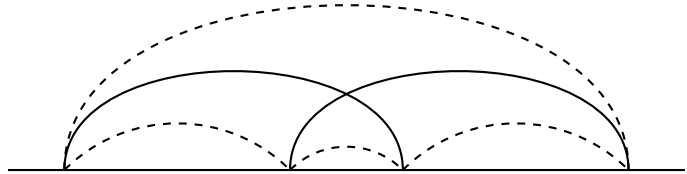

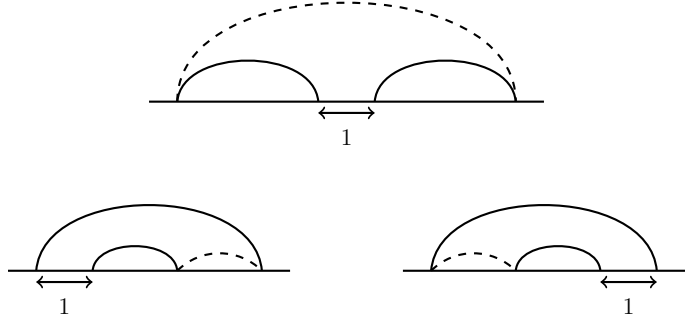
\begin{figure}[!ht]
\begin{center}
\begin{tikzpicture}[thick,scale=0.75, every node/.style={scale=0.75}]

\draw (3.5,4) -- (10.5,4);

\path[out=85,in=95] (4,4) edge (6.5,4);
\path[out=85,in=95] (7.5,4) edge (10,4);

\path[out=90,in=90,dashed] (4,4) edge (10,4);

\draw[<->] (6.5,3.8)-- (7.5,3.8);

\node (arrow) at (7,3.8) {};
\node [below] at (arrow.south) {1};

\draw (1,1) -- (6,1);

\draw (8,1) -- (13,1);

\path[out=85,in=95] (1.5,1) edge (5.5,1);
\path[out=85,in=95] (2.5,1) edge (4,1);
\path[out=45,in=135,dashed] (4,1) edge (5.5,1);

\draw[<->] (1.5,0.8)-- (2.5,0.8);

\node (arrow1) at (2,0.8) {};
\node [below] at (arrow1.south) {1};

\path[out=85,in=95] (8.5,1) edge (12.5,1);
\path[out=85,in=95] (10,1) edge (11.5,1);
\path[out=45,in=135,dashed] (8.5,1) edge (10,1);

\draw[<->] (11.5,0.8)-- (12.5,0.8);
\node (arrow2) at (12,0.8) {};
\node [below] at (arrow2.south) {1};

\end{tikzpicture}
\end{center}
\caption{The Ptolemy arcs of class II.}
\label{fig:modified-ptolemy}
\end{figure} 

Note that Ptolemy arcs of class II are always admissible. However, the same does not hold for Ptolemy arcs of class I. In fact, there is an extension between two strictly crossing arcs if and only if at least one of the Ptolemy arcs is admissible. In this case, we can say more, which is stated in the following lemma, whose proof is obvious. 

\begin{lemma}\label{lem:admissibleclassI}
Given two strictly crossing arcs in $\Tw$ and $\{\arc{x}, \arc{y}\}$ two of the corresponding Ptolemy arcs (of class I) which are opposite to each other, we have $\arc{x}$ admissible if and only if $\arc{y}$ is admissible. 
\end{lemma}

Now we gather some useful results from \cite{CSP}.

\begin{proposition}[{\cite[Proposition 7.8 and Lemma 7.4]{CSP}}]\label{prop:extensions}
Let $a, b \in \ind{\Tw}$ be such that \linebreak $\Ext^1 (b,a) \neq 0$. Let $e$ be the middle term of the corresponding non-split triangle, and consider the arcs:
\[
\arc{e}_1^+ \coloneqq (s(\arc{a}),t(\arc{b})), \quad \arc{e}_2^+ \coloneqq (s(\arc{b}), t(\arc{a})), \quad
\arc{e}_1^- \coloneqq (s(\arc{b}),s(\arc{a})), \quad \arc{e}_2^- \coloneqq (t(\arc{b}), t(\arc{a})).
\]
Then, besides the case when $b = \Sigma a$ and $e = 0$, we are in one of the following cases:
\begin{compactitem}
\item $\arc{a}, \arc{b}$ strictly cross such that $t(\arc{b}) < t(\arc{a})$, $e = e_1^+ \oplus e_2^+$. 
\item $\arc{a}, \arc{b}$ strictly cross such that $t(\arc{a}) < t(\arc{b})$, $e = e_1^- \oplus e_2^-$.
\item $\arc{a}, \arc{b}$ are neighbouring arcs such that $t(\arc{a}) = s(\arc{b})+1$, $e = e_1^+$.
\item $\arc{a}, \arc{b}$ are neighbouring arcs such that $t(\arc{a}) = t(\arc{b})+1$, $e = e_1^-$.
\item $\arc{a}, \arc{b}$ are neighbouring arcs such that $s(\arc{a}) = s(\arc{b})+1$, $e = e_2^-$. 
\end{compactitem}
\end{proposition}

A priori the arcs $\arc{e}_1^+, \arc{e}_1^-, \arc{e}_2^+$ and $\arc{e}_2^-$ may not be admissible. But, in each case, the ones mentioned are admissible.

\begin{theorem}[{\cite[Theorem 7.2]{CSP}}]\label{prop:computationofextensions}
Let $\sX$ be a subcategory of $\sT_w$ and $\tX$ be the corresponding collection of admissible arcs. Then the objects of $\ind{\extn{\sX}}$ correspond to the arcs of the closure of $\tX$ under admissible Ptolemy arcs of class $I$ and $II$.
\end{theorem}

\section{Classification of $w$-simple-minded systems in $\Tw$}

In this section, we shall give a classification of $w$-simple-minded systems in $\Tw$ in terms of admissible arcs in the $\infty$-gon. Firstly, we have the following generalisation of the fact that every simple-minded system is a Riedtmann configuration. Recall that in a distinguished triangle $\trilabels{x}{y}{z}{}{g}{}$, $x$ is called the {\it cocone of $g$} and it is denoted by $\cocone{g}$. 

\begin{lemma}\label{lem:smsimpliesRiedtmann}
Given a triangulated category $\sT$ with a Serre functor and $w \leq -1$, any $w$-simple-minded system in $\sT$ is a $w$-Riedtmann configuration. 
\end{lemma}
\begin{proof}
Suppose $\sS$ is a $w$-simple-minded system, and let $t \in \sT$. Then there is $n \geq 1$ such that $ t \in (\sX)^{\oplus}_n$, with $\sX = \bigcup\limits_{i=w+1}^0 \Sigma^i \sS$. We shall prove that there is $x \in \sX$ such that $\Hom (t,x) \neq 0$ by induction on $n$. If $n = 1$, then $t \in \sX$, and the result is clear. If $n > 1$, by Lemma \ref{lem:closedsummands} (1), we have a triangle $\trilabels{y}{t \oplus t^\prime}{x}{}{\rowmat{f}{g}}{}$, where $y \in (\sX)_{n-1}$ and $x \in \sX$. If $f = 0$, then $y = \cocone{g} \oplus t$, and so $t$ is a summand of $y$. Since, by definition, $(\sX)_{n-1}$ is closed under direct summands, it follows that $t \in (\sX)_{n-1}$, and the result holds by induction hypothesis. If $f \neq 0$, then $\Hom (t, x) \neq 0$, as required. Therefore, $\sS$ is a right $w$-Riedtmann configuration. 

Note that if $\sS$ is a $w$-simple-minded system, then so is $\Sigma^{-w-1} \sS$. We can use this fact to prove, in a similar way, that $\sS$ is also a left $w$-Riedtmann configuration in $\sT$.
\end{proof}

The main result of this section is as follows, whose proof is split up into three lemmas.

\begin{theorem}\label{thm:classificationsms}
A collection of indecomposable objects $\sS$ in $\Tw$ is a $w$-simple-minded system if and only if the collection $\tS$ of the corresponding admissible arcs satisfies the following conditions:
\begin{compactenum}
\item There are no crossings.
\item Every arc in $\tS$ has precisely $|w|-1$ inner-isolated vertices. 
\item There are no overarcs, i.e. every arc in $\tS$ has an overarc. 
\end{compactenum}
\end{theorem}

\begin{lemma}
Let $\sS$ be a $w$-simple-minded system in $\Tw$. Then, there are no strict-crossings and every arc in $\tS$ has precisely $|w|-1$ inner-isolated vertices. 
\end{lemma}
\begin{proof}
It follows immediately from Lemma \ref{lem:smsimpliesRiedtmann} and Theorem \ref{thm:Hom-config}.
\end{proof}

\begin{lemma}\label{lem:arcssms}
In a $w$-simple-minded system in $\Tw$, every arc has an overarc. 
\end{lemma}
\begin{proof}
Let $\sS$ be a $w$-simple-minded system in $\Tw$ and $\sX = \bigcup\limits_{i=w+1}^0 \Sigma^i \sS$. Suppose there is an outer-arc $\alpha$ in $\tS$. We claim that the following statements hold for every $n \geq 1$:

$(i)_n$: There is no overarc $\sigma$ of $s(\alpha)$ in $(\sX)_n$ such that $(s(\sigma), s(\alpha))$ is admissible, and

$(ii)_n$: There is no arc $\sigma$ in $(\sX)_n$ such that $t(\sigma) = s(\alpha)$. 

The proof is done by induction on $n$. 

$(i)_1$:  Since $s(\alpha)$ has no overarcs in $\tS$, if $\sigma \in \tX$ is such that $\sigma$ is an overarc of $s(\alpha)$, then $s(\sigma)-s(\alpha) \leq |w|-1$, and so $(s(\sigma), s(\alpha))$ is not admissible. 

$(ii)_1$:  It is clear that, given an arc $\sigma \in \sS$ to the right of $s(\alpha)$, there is no $i \in \{w+1, \ldots, -1\}$ such that $t(\Sigma^i \sigma) = s(\alpha)$. If $\sigma$ is an arc in $\sS$ such that $s(\sigma) \leq s(\alpha)$, then $s(\alpha) - t(\Sigma^i \sigma) \geq 1$, for every $1+w \leq i \leq -1$, and so there is no arc in $\tX$ whose target is $s(\alpha)$. 

Suppose $(i)_n$ and $(ii)_n$ hold. Firstly, let us prove $(i)_{n+1}$. Suppose, for a contradiction, that there is an overarc $\sigma$ of $s(\alpha)$ in $(\tX)_{n+1}$ such that $(s(\sigma), s(\alpha))$ is admissible. Then, by Theorem \ref{prop:computationofextensions}, $\sigma$ must be obtained from arcs $\arc{a}$ and $\arc{b}$ in $(\tX)_n$ in one of the ways shown in Figure \ref{fig:ptolemysigma}.

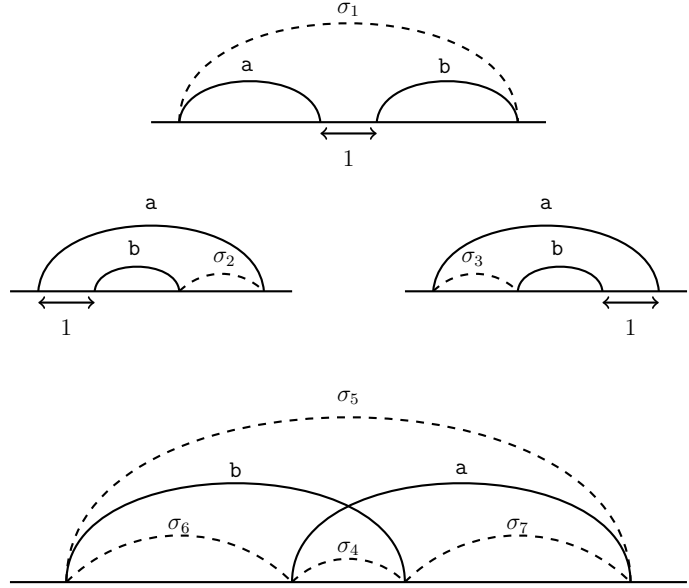
\begin{figure}[!ht]
\begin{center}
\begin{tikzpicture}[thick,scale=0.75, every node/.style={scale=0.75}]


\draw (4.5,4) -- (11.5,4);

\path[out=85,in=95] (5,4) edge (7.5,4);
\path[out=85,in=95] (8.5,4) edge (11,4);

\path[out=90,in=90,dashed] (5,4) edge (11,4);

\draw[<->] (7.5,3.8)-- (8.5,3.8);

\node (arrow) at (8,3.8) {};
\node [below] at (arrow.south) {1};

\node (a) at (6.2,5) {};
\node at (a) {$\mathtt{a}$};

\node (b) at (9.7,5) {};
\node at (b) {$\mathtt{b}$};

\node (1) at (8,6) {};
\node at (1) {$\sigma_1$};

\draw (2,1) -- (7,1);

\draw (9,1) -- (14,1);

\path[out=85,in=95] (2.5,1) edge (6.5,1);
\path[out=85,in=95] (3.5,1) edge (5,1);
\path[out=45,in=135,dashed] (5,1) edge (6.5,1);

\draw[<->] (2.5,0.8)-- (3.5,0.8);

\node (arrow1) at (3,0.8) {};
\node [below] at (arrow1.south) {1};

\node (b) at (4.2,1.5) {};
\node [above] at (b) {$\mathtt{b}$};

\node (a) at (4.5,2.3) {};
\node [above] at (a) {$\mathtt{a}$};

\node (2) at (5.8,1.3) {};
\node [above] at (2) {$\sigma_2$};

\path[out=85,in=95] (9.5,1) edge (13.5,1);
\path[out=85,in=95] (11,1) edge (12.5,1);
\path[out=45,in=135,dashed] (9.5,1) edge (11,1);

\draw[<->] (12.5,0.8)-- (13.5,0.8);
\node (arrow2) at (13,0.8) {};
\node [below] at (arrow2.south) {1};

\node (b) at (11.7,1.5) {};
\node [above] at (b) {$\mathtt{b}$};

\node (a) at (11.5,2.3) {};
\node [above] at (a) {$\mathtt{a}$};

\node (3) at (10.2,1.3) {};
\node [above] at (3) {$\sigma_3$};

\end{tikzpicture}
\end{center}
\begin{center}
\begin{tikzpicture}[thick,scale=0.75, every node/.style={scale=0.75},out=90,in=90]

\draw (0,1) -- (12,1);

\path (1,1) edge (7,1);
\path (5,1) edge (11,1);

\node (b) at (4,3) {};
\node at (b) {$\mathtt{b}$};

\node (a) at (8,3) {};
\node at (a) {$\mathtt{a}$};

\path[out=45,in=135,dashed] (1,1) edge (5,1);
\path[dashed] (1,1) edge (11,1);
\path[out=45,in=135,dashed] (5,1) edge (7,1);
\path[out=45,in=135,dashed] (7,1) edge (11,1);

\node (4) at (6,1.6) {};
\node at (4) {$\sigma_4$};

\node (5) at (6,4) {};
\node [above] at (5) {$\sigma_5$};

\node (6) at (3,2) {};
\node at (6) {$\sigma_6$};

\node (7) at (9,2) {};
\node at (7) {$\sigma_7$};

\end{tikzpicture}
\end{center}
\caption{$\arc{a}, \arc{b} \in (\tX)_n$ and $\sigma_i \in (\tX)_{n+1}$.}
\label{fig:ptolemysigma}
\end{figure}

\Case{$\sigma = \sigma_1$} If $s(\alpha) = t(\arc{b})$, then $\arc{b}$ contradicts $(ii)_n$. If $s(\alpha) = s(\arc{a})$, then $(s(\sigma), s(\alpha))$ cannot be admissible, since $\arc{b}$ is admissible. If $\arc{b}$ is an overarc of $s(\alpha)$, then $\arc{b}$ contradicts $(i)_n$. Finally, if $\arc{a}$ is an overarc of $s(\alpha)$, then $\arc{a}$ contradicts $(i)_n$, since we would have an admissible arc $(s(\arc{a}),s(\alpha))$. 

\Case{$\sigma = \sigma_2$ or $\sigma = \sigma_7$} The arc $\arc{a}$ clearly contradicts $(i)_n$. 

\Case{$\sigma = \sigma_3$} In this case, the arc $\arc{a}$ contradicts $(i)_n$ since, if $(s(\sigma), s(\alpha))$ is admissible then so is $(s(\arc{a}), s(\alpha))$.

\Case{$\sigma = \sigma_4$} The arc $\arc{b}$ clearly contradicts $(i)_n$. 

\Case{$\sigma = \sigma_5$} If $\arc{a}$ is an overarc of $s(\alpha)$, then it contradicts $(i)_n$. If $s(\alpha) = t(\arc{a})$, then $\arc{a}$ contradicts $(ii)_n$. Hence, we have $t(\arc{b}) < s(\alpha) < t(\arc{a})$. But, since $(s(\arc{a}), s(\alpha))$ is admissible, so is $(s(\arc{b}), s(\alpha))$, and so $\arc{b}$ contradicts $(i)_n$.

\Case{$\sigma = \sigma_6$} If $(s(\sigma),s(\alpha))$ is admissible then so is $(s(\arc{b}),s(\alpha))$, contradicting $(i)_n$. Therefore, $(i)_{n+1}$ holds. 

Lastly, we prove $(ii)_{n+1}$. An arc $\sigma$ in $(\tX)_{n+1}$ whose target is $s(\alpha)$ can be obtained as in Figure \ref{fig:ptolemysigma}. However, each case except $\sigma_2$ and $\sigma_7$ contradict $(ii)_n$, and the cases $\sigma_2$ and $\sigma_7$ contradict $(i)_n$.

Hence $\bigcup\limits_{n \geq 1} (\tX)_n$ does not contain, for instance, admissible arcs whose target is $s(\alpha)$, contradicting the fact that $\sS$ is a $w$-simple-minded system.
\end{proof}

\begin{lemma}\label{lemma:overarcimpliessimpleminded}
Any $w$-Riedtmann configuration in $\Tw$ with no outer-arcs is a $w$-simple-minded system in $\Tw$. 
\end{lemma}
\begin{proof}
Let $\sS$ be a $w$-Riedtmann configuration in $\Tw$ with no outer-arcs. Then it is clear that every vertex of the $\infty$-gon has an overarc in $\tS$. 

If $\bigcup\limits_{n \geq 1} (\sX)_n^\oplus$ contains every admissible arc of minimum length, then we have $\bigcup\limits_{n \geq 1} (\sX)^\oplus_n = \Tw$, and so $\sS$ is a $w$-simple-minded system. Hence, given $x \in \bZ$, we want to show that $(x, x+w) \in \bigcup\limits_{n \geq 1} (X)_n^\oplus$. By above, $x+w$ is an inner-vertex of an arc $\arc{a}_0$ in $\tS$. 

\Case{$\arc{a}_0$ is not an overarc of $x$} Then $x+w < s(\arc{a}_0) < x$. Let $\{\mathtt{a}_i\}_{i=1}^r$ be the sequence of inner-arcs of $\arc{a}_0$ in $\tS$, ordered from left to right. Note that we must have $s(\arc{a}_r) \leq x+w$. Now consider the sequence of arcs $\{\Sigma^{k_i} \arc{a}_i\}_{i=0}^r$ such that $t(\Sigma^{k_i} \arc{a}_i) = t(\Sigma^{k_{i+1}} \arc{a}_{i+1})-1$, $s(\Sigma^{k_0} \arc{a}_0) = x$ and $s(\Sigma^{k_r} \arc{a}_r) = x+w$. Recall that, since $\sS$ is a $w$-Riedtmann configuration, any arc in $\tS$ has precisely $|w|-1$ inner-isolated vertices. Therefore, $w+1 \leq k_i \leq 0$, for all $i = 0, \ldots, r$, and so the arcs $\{\Sigma^{k_i} \arc{a}_i\}_{i=0}^r$ lie in $\sX = \bigcup\limits_{j= w+1}^0 \Sigma^i \sS$. It follows by Theorem \ref{prop:computationofextensions} that $(x, s(\arc{a}_i)) \in (\tX)^\oplus_{1+i}$, for $i = 1, \ldots, r$. In particular, $(x, x+w) \in (\tX)^\oplus_{1+r} \subseteq \bigcup\limits_{n \geq 1} (\tX)^\oplus_n$. 

\Case{$\mathtt{a}_0$ is an overarc of $x$} Let $\{\arc{a}_i\}_{i= 0}^s$, be the sequence of inner-arcs of $\arc{a}_0$ in $\tS$, ordered from left to right, and let $r < s$ such that $s(\arc{a}_r) \leq x +w$ and $t(\arc{a}_{r+1}) > x + w$. Note that $t(\arc{a}_{r+1}) \leq x$, as $\arc{a}_0$ has precisely $|w|-1$ inner-isolated vertices. Now consider the sequence of arcs $\{ \Sigma^{k_i} \mathtt{a}_i\}_{i= 0}^s$ such that $s(\Sigma^{k_r} \arc{a}_r) = x+w$, $t(\Sigma^{k_{r+1}} \arc{a}_{r+1}) = x$, $s(\Sigma^{k_i} \arc{a}_i) = t(\Sigma^{k_{i+1}} \arc{a}_{i+1}) -1$, for $i \in \{1, \ldots, s-1\} \setminus \{r\}$, $s(\Sigma^{k_0} \arc{a}_0) = s(\Sigma^{k_s} \arc{a}_s) +1$ and $t(\Sigma^{k_0} \arc{a}_0) = t(\Sigma^{k_1} \arc{a}_1) - 1$. Note that, as in Case 1, these arcs lie in $\tX$. Now, by Theorem \ref{prop:computationofextensions}, we have $(s(\Sigma^{k_0} \arc{a}_0), s(\Sigma^{k_i} \arc{a_i})) \in (\sX)^\oplus_{i+1}$, for $1 \leq i \leq r$, and $(t(\Sigma^{k_j} \arc{a}_j), x+w) \in (\tX)^\oplus_{r+2+s-j}$, for $r+1 \leq j \leq s$. In particular, $(x,x+w) \in (\tX)^\oplus_{s+1} \subseteq \bigcup\limits_{n \geq 1} (\tX)^\oplus_n$, and we are done.
\end{proof}

\begin{remark}\label{rem:Dugasquestion}
In $\sT_{-1}$, any $(-1)$-Riedtmann configuration with an outer-arc is an example of a $(-1)$-Riedtmann configuration which is not a simple-minded system. This gives a negative answer to Dugas' question whether Riedtmann configurations coincide with simple-minded systems \cite{Dugas12}. 
\end{remark}

\begin{remark}\label{rem:simplemindedcollections}
Let $\sS$ be a $w$-orthogonal set of indecomposable objects in a triangulated category $\sT$. Note that condition (d) (cf. definition of $w$-simple-minded system) is not equivalent to saying $\thick{}{\sS} = \sT$. Indeed, whilst $\bigcup\limits_{n \geq 1} (\sX)_n^\oplus \subseteq \thick{}{\sS}$, where $\sX = \bigcup\limits_{i=w+1}^0 \Sigma^i \sS$, the converse is, in general, not true. For instance, given any collection $\sS^\prime$ of arcs in $\Tw$ with no crossings and with at least one arc of length $|w|$, we have $\thick{}{\sS^\prime} = \Tw$. Any $w$-Hom configuration in $\Tw$ would satisfy this property. However, if there is one arc in $\sS^\prime$ with more than $|w|-1$ inner-isolated vertices, then $\sS^\prime$ is not a $w$-Riedtmann configuration, and therefore it does not satisfy condition (d) either.
\end{remark}

\subsection{Noncrossing partitions in $\sT_{-1}$}

We can give a different combinatorial classification of $(-1)$-simple-minded systems in $\sT_{-1}$, in terms of noncrossing partitions of $\mathbb{Z}$.

\begin{definition}[{\cite[Definition 7.1]{CS3}}]
A \textit{noncrossing partition of $\mathbb{Z}$} is a set of pairwise disjoint non-empty subsets, called \textit{blocks}, of vertices of the $\infty$-gon, whose union is all of $\mathbb{Z}$, and such that no two blocks cross each other, i.e. if $a$ and $b$ belong to one block and $x$ and $y$ to another, we cannot have $a < x < b < y$. 
\end{definition}

In \cite[Propositions 7.2 and 7.3]{CS3}, (-1)-Hom- and (-1)-Riedtmann-configurations in $\sT_{-1}$ were also classified in terms of noncrossing partitions of $\bZ$. 

We will now recall the maps used in this classification, as we will also use them to classify simple-minded systems in $\sT_{-1}$. 

Let $\mathbb{Z}^\prime\coloneqq \{2n+0.5 \mid n \in \bZ\}$. Given a $(-1)$-Hom configuration $\sS$ in $\sT_{-1}$, we define $\mathsf{nc}(\sS)$ to be the partition $\mathcal{P} = \{B_1, B_2, \ldots\}$ of $\mathbb{Z}^\prime$, where the vertices in a block $B =\{x_1, x_2, \ldots\}$ are written in numerical order, i.e. $x_1 < x_2 < \ldots$, and are defined recursively as follows: given $x_i$, if $x_i + 0.5$ is the target of an arc $\arc{a}$ in $\sS$, then $x_{i+1}\coloneqq s(\arc{a})+0.5$. Otherwise, we take $x_i$ to be the last element of its block. This defines a map $\mathsf{nc}$ from the class of $(-1)$-Hom-configurations (seen as collections of admissible arcs satisfying the conditions in Theorem \ref{thm:Hom-config}) to the class of noncrossing partitions of the integers.

Given a noncrossing partition $\cP$ of $\bZ^\prime$, its \textit{Kreweras complement} $\sK (\cP)$ is the unique maximal partition of $\bZ^{\prime \prime} \coloneqq \{2n-0.5 \mid n \in \bZ\}$ such that $\cP \cup \sK (\cP)$ is a noncrossing partition of $\bZ^\prime \cup \bZ^{\prime \prime}$.

The following proposition characterizes simple-minded systems in $\sT_{-1}$ in terms of noncrossing partitions.

\begin{proposition}
A $(-1)$-Riedtmann configuration $\sS$ in $\sT_{-1}$ is a simple-minded system if and only if $\mathsf{nc}(\sS)$ and $\sK \mathsf{nc}(\sS)$ have no infinite block.  
\end{proposition}
\begin{proof}
Suppose $\sS$ is a simple-minded system and assume that $\mathsf{nc}(\sS)$ has an infinite block $B$. Let $x \in B$. Since $\sS$ is a $(-1)$-Riedtmann configuration, $B$ has no upper bound (cf. \cite[Proposition 7.3]{CS3}), and so there is $y > x \in B$. We may assume that $y$ is the smallest element satisfying this condition, and therefore $\arc{a} \coloneqq (y-0.5,x+0.5) \in \sS$. By hypothesis, $\arc{a}$ has an overarc, which implies that $B$ is finite, a contradiction. The proof that $\sK \mathsf{nc}(\sS)$ has no infinite block is analogous. 

Now suppose $\mathsf{nc}(\sS)$ and $\sK \mathsf{nc}(\sS)$ have no infinite block, and let $\arc{a} = (t,u)$ be an arc in $\sS$. Then $u-0.5$ and $t+0.5$ belong to the same block $B$ of either $\mathsf{nc}(\sS)$ or $\sK \mathsf{nc}(\sS)$. By hypothesis, $B$ is finite. Let $x$ and $y$ be the minimum and maximum element of $B$, respectively. Then we must have $\arc{b} \coloneqq (y+0.5, x-0.5) \in \sS$. Since $y \geq t+0.5$ and $x \leq u-0.5$, $\arc{b}$ is an overarc of $\arc{a}$, which proves that $\sS$ is a simple-minded system.   
\end{proof}

\section{Mutations in $\Tw$}

We now study the behaviour of mutations of $w$-Hom configurations, $w$-Riedtmann configurations and $w$-simple-minded systems in $\Tw$. We shall see that, when $w \leq -2$, the behaviour of the latter two classes of objects is very reminiscent of that of (weakly) $(|w|-1)$-cluster-tilting subcategories in $\sT_{|w|}$ studied in \cite{HJ2}.

\subsection{Number of replacements} Firstly, we shall use the combinatorial classifications to identify the number of replacements of a given indecomposable object.
 
\begin{proposition}\label{prop:mutationsT-m}
Let $\sS$ be a $w$-Hom configuration in $\Tw$, $k$ be the number of outer-isolated vertices, and $\arc{s}$ an arc in $\tS$.
\begin{compactenum}
\item If $\tU$ is a set of admissible arcs not in $\tS \setminus \{\arc{s}\}$ such that $\tS \setminus \{\arc{s}\} \cup \tU$ is a $w$-Hom configuration, then $\tU = \{\arc{s}^*\}$ for a single admissible arc $\arc{s}^*$. 
\item If $\arc{s}$ has an overarc in $\tS$, then there are $|w|$ choices for $\arc{s}^*$.
\item If $\arc{s}$ has no overarc in $\tS$, then there are $k+1$ choices for $\arc{s}^*$.   
\end{compactenum}
\end{proposition}
\begin{proof}
Let $\tS^\prime \coloneqq \tS \setminus \{\arc{s}\} \cup \tU$ be a $w$-Hom configuration, and $\arc{s}^* \in \tU$.

Suppose $\arc{s} = (t,u)$ has an overarc in $\tS$, and let $\arc{a}$ be the smallest one. Since, $\sS$ is a $w$-Hom-configuration, $\arc{s}$ and $\arc{a}$ have each precisely $|w|-1$ inner-isolated vertices. Therefore, if we remove $\arc{s}$ from $\tS$, there are $2|w|$ isolated vertices whose smallest overarc is $\arc{a}$. Let us label these vertices in numerical order by $x_1, x_2, \ldots, x_{2|w|}$. It is easy to check that the $(x_{|w|+i}, x_i)$, with $i = 1, \ldots, |w|$, are the admissible arcs that can lie in $\tU$. Note that a pair of such arcs would cross, hence $\tU$ must be a singleton. Moreover, there are $|w|$ such admissible arcs, and so (2) follows.   

Now suppose $\arc{s} = (t,u)$ has no overarc in $\tS$. Let $v_1, v_2, \ldots, v_k$ be the outer-isolated vertices in $\tS$, $x_1, x_2, \ldots, x_{|w|-1}$ be the inner-isolated vertices of $\arc{s}$ in $\tS$, and suppose they are ordered numerically. Then $V  \coloneqq \{x_1, \ldots, x_{|w|-1}\} \cup \{v_1, \ldots, v_k\} \cup \{u, t\}$ is the set of outer-isolated vertices in $\tS \setminus \{\arc{s}\}$. Given that $\tS^\prime$ is a $w$-Hom configuration, the endpoints of $\arc{s}^*$ must lie in $V$.

\Case{$\arc{s}^* = (x_j, x_i)$ or $\arc{s}^* = (t, x_i)$ or $\arc{s}^* = (x_i, u)$} 
In this case, $\arc{s}^*$ would have fewer than $|w|-1$ inner-isolated vertices in $\tS^\prime$, contradicting the fact that $\tS^\prime$ is a $w$-Hom configuration. 

\Case{$\arc{s}^* = (v_j, v_i)$}
There are no strict-crossings between $\arc{s}^*$ and any arc in $\tS$, contradicting the maximality of $\tS$. 

\Case{$\arc{s}^* = (v_j,t)$ or $\arc{s}^* = (u, v_j)$} 
Since $\sS^\prime$ is a $w$-Hom configuration, $\arc{s}^*$ has precisely $|w|-1$ inner-isolated vertices. Note that these vertices must be outer-isolated vertice in $\tS$. Therefore, we must have $k = |w|$ and either $t < v_1$ and $\arc{s}^* = (v_m, t)$, or $v_m < u$ and $\arc{s}^* = (u,v_1)$.

\Case{$\arc{s}^* = (t, v_i)$}
On one hand, since $\arc{s}$ and $\arc{s}^*$ are admissible, there are $l (|w|+1)- 1$ vertices in $]v_i,u[$, for some $l \geq 1$. On the other hand, there is no arc crossing $(t,v_i)$ in $\tS$, since $\tS^\prime$ is a $w$-Hom configuration. These two facts imply that $\tS$ has $|w|$ outer-isolated vertices in $]v_i,u[$, and therefore there are at least $|w|+1$  outer-isolated vertices in $\tS$, a contradiction. 

\Case{$\arc{s}^* = (v_j, u)$}
The same argument as in the previous case holds. 

\medskip

Therefore, each arc in $\tU$, besides $\arc{s}$, must be of the form $(x_i,v_j)$, $(v_j,x_i)$, $(v_{|w|},t)$ in the case when $k = |w|$ and $t < v_1$ or $(u,v_1)$ in the case when $k = |w|$ and $v_{|w|} < u$. 

Let $v_j$ be the largest outer-isolated vertex in $\tS$ which is smaller than $u$ (if it exists). Note that the number of vertices in the intervals (whenever they make sense): 
\[
]v_i,v_{i+1}[, ]v_j,u[, ]u,x_1[, ]x_{l},x_{l+1}[, ]x_{|w|-1},t[, ]t,v_{j+1}[,
\]
with $i \in \{1, \ldots, j-1, j+1, \ldots, k-1\}$ and $l \in \{1, \ldots, |w|-2\}$, is a multiple of $|w|+1$. 

Hence, if $v_j$ does not exist, or in other words $t < v_1$, the only admissible arcs incident with $V$ are of the form $(v_i,x_i)$, for $i = 1, \ldots, k$. Here we let $x_{|w|} \coloneqq t$, in the case when $k = |w|$. Similarly, if $j = k$, i.e. $v_k < u$, the only admissible arcs incident with $V$ are $(x_{|w|-i}, v_{k-i+1})$ for $i = 1, \ldots, k$. Here we let $x_0 := u$, in the case when $k = |w|$. Finally, if $1 \leq j \leq k-1$, then the only admissible arcs which are incident with $V$ are: $(x_{|w|-i},v_{j-i+1})$, with $i = 1, \ldots, j$, and $(v_{j+i}, x_i)$, with $i = 1, \ldots, k-j$. 

In each case we have precisely $k$ admissible arcs. If $\tA$ denotes the set of these arcs together with $\arc{s}$, note that any pair of arcs in $\tA$ is strict-crossing. Therefore $\tU$ is again a singleton and there are $k +1$ choices for $\arc{s}^*$. 
\end{proof}

\begin{remark}
One can deduce from the proof of Proposition \ref{prop:mutationsT-m} that the number of outer-isolated vertices is preserved under mutation. Therefore, the mutation graph of $w$-Hom configurations is disconnected; cf.~\cite{Aihara-Iyama}.
\end{remark}

Given a $w$-Hom configuration $\sS$ in $\Tw$, and $s \in \sS$, we say that a $w$-Hom configuration of the form $\sS \setminus \{s\} \cup \{s^*\}$, with $s^* \neq s$, is a \textit{mutation of $\sS$ at $s$}. The object $s^*$ is called a \textit{replacement of $s$}. 

\begin{corollary}\label{cor:replacementssms}
Let $\sS$ be a $w$-simple-minded system in $\Tw$, and $s \in \sS$. Then $s$ has precisely $|w|-1$ replacements.
\end{corollary}
\begin{proof}
This follows immediately from Proposition \ref{prop:mutationsT-m} and from the fact that there are no outer-arcs in $\tS$.
\end{proof}

\begin{corollary}\label{cor:replacementsl}
Given $0 \leq l \leq |w|-1$, there is a $w$-Riedtmann configuration $\sS$ in $\Tw$ and $s \in \sS$ such that $s$ has precisely $l$ replacements.
\end{corollary}
\begin{proof}
Take $\sS$ to be a $w$-Hom configuration with $l$ outer-isolated vertices and any outer-arc $\arc{s}$ in $\tS$. The result then follows from Proposition \ref{prop:mutationsT-m}.  
\end{proof}

\begin{corollary}\label{cor:replacementsriedtmann}
Let $\sS$ be a $w$-Riedtmann configuration in $\Tw$ and $s \in \sS$. Then $s$ has at most $|w|-1$ replacements. 
\end{corollary} 

\begin{remark}
There are precisely two classes of configurations in $\Tw$ with a `nice mutation behaviour', in the sense that the number of replacements of a given object is exactly $|w|-1$, namely the $w$-simple-minded systems and the $w$-Riedtmann configurations with $|w|-1$ outer-isolated vertices. 
\end{remark}

The purpose of the remainder of this paper is to convey to the reader that the natural analogue of $(|w|-1)$-cluster tilting subcategories in $\sT_{|w|}$ ($w \leq -2$) in negative Calabi-Yau triangulated categories is played by $w$-simple-minded systems. Informally speaking, we shall see that even though a $w$-Riedtmann configuration $\sS$ with $|w|-1$ outer-isolated vertices has a similar mutation behaviour to that of a $w$-simple-minded system, the mutations of $\sS$ do not have such a nice representation-theoretic description. 
 
\subsection{Functorial finiteness}
The extension closure of a $w$-Hom configuration will play a key role in its mutation theory. In this subsection we shall prove that the only $w$-Hom configurations for which the extension closure is functorially finite are the $w$-simple-minded systems. 

Let $\sC$ be any category and $\sA$ be a subcategory. A morphism $f\colon a \to c$, with $a \in \sA$, is called a \emph{right $\sA$-approximation} of $c$ if every map $g\colon a^\prime \to c$, with $a^\prime \in \sA$, factors through $f$. We say that $\sA$ is a \emph{contravariantly finite subcategory of $\sC$} if every object of $\sC$ admits a right $\sA$-approximation. There are dual notions of \emph{left $\sA$-approximation} and \emph{covariantly finite}. If $\sA$ is both contra- and covariantly finite, $\sA$ is called \emph{functorially finite}.

Let $\tA$ be a collection of admissible arcs in $\Tw$. An integer $t$ is called a \emph{left fountain} of $\tA$ if $\tA$ contains infinitely many arcs of the form $(t,s)$ (with $s < t$). Dually, one defines a \emph{right fountain} of $\tA$; a \emph{fountain} of $\tA$ is both a left fountain and a right fountain.

\begin{theorem}[{\cite[Corollary 6.2]{CSP}}]\label{cor:contravariant-finiteness}
Let $\sX$ be a subcategory of $\Tw$ and $\tX$ be the set of admissible arcs corresponding to $\ind{\sX}$. Then $\sX$ is 
\begin{compactenum}
\item contravariantly finite in $\Tw$ if and only if every left fountain in $\tX$ is also a right fountain,
\item covariantly finite in $\Tw$ if and only if every right fountain in $\tX$ is also a left fountain. 
\end{compactenum}
\end{theorem}

The main result of this subsection is as follows.

\begin{theorem}\label{thm:functfinite}
Let $\sX$ be a $w$-orthogonal collection of objects in $\Tw$ and $\tX$ be the corresponding collection of admissible arcs. Suppose there are finitely many outer-isolated vertices in $\tX$. Then the following assertions are equivalent:
\begin{compactenum}
\item $\extn{\sX}$ is functorially finite;
\item There are no outer-arcs in $\tX$. 
\end{compactenum}
\end{theorem}

Before proving this theorem, we need a pair of lemmas.

Recall that the extension closure of any subcategory of $\Tw$ corresponds to taking the closure under Ptolemy arcs of classes I and II. The following lemma, which will be used to prove Theorem \ref{thm:functfinite}, is a strengthening of this result, in the case when the subcategory in question is $w$-orthogonal.

\begin{lemma}\label{cor:extensionclosurePtolemyII}
Let $\sX$ be a $w$-orthogonal collection of indecomposable objects in $\Tw$. Then the objects of $\ind{\extn{\sX}}$ correspond to the arcs of the closure of $\tX$ under Ptolemy arcs of class II. 
\end{lemma}
\begin{proof}
Suppose $x \in \ind{\extn{\sX}}$. By Lemma \ref{lem:closedsummands} (2), we have $x \in \sX * (\sX)_m$, for some $m \geq 1$. In particular, there is a distinguished triangle $\tri{a}{x}{b}$ with $a \in \sX$ and $b \in (\sX)_m$. Since $\sX$ is just a collection of indecomposable objects regarded as a full subcategory, $a$ is indecomposable. 

\Case{$b$ is indecomposable} Since $\Ext^1 (b,a) \neq 0$, we have that $\arc{a}$ and $\arc{b}$ are either neighbouring arcs or they strictly cross (see Proposition \ref{prop:extensions}). If $\arc{a}, \arc{b}$ strictly cross, then by Lemma \ref{lem:admissibleclassI} and Proposition \ref{prop:extensions}, $x$ is not indecomposable. This contradicts the hypothesis, and therefore $\arc{a}$ and $\arc{b}$ are neighbouring arcs, making $\arc{x}$ a Ptolemy arc of class II, as required.

Now assume $b$ is not indecomposable and write $b = \sum\limits_{i=1}^k b_i$, where each $b_i$ is indecomposable. We can assume, without loss of generality, that $b_i \neq \Sigma a$, for all $i$. 

\Case{$k \geq 3$} If $\{b_i \to \Sigma a \mid i= 1, \ldots, k\}$ is not factorisation-free (see definition in \cite[Definition 4.2]{CSP}), then it follows from \cite[Corollary 4.4]{CSP} that $x$ is not indecomposable, which is a contradiction. Hence, $\{b_i \to \Sigma a \mid i= 1, \ldots, k\}$ is factorisation-free, and we can apply \cite[Proposition 4.12]{CSP}. However, each of the $x_i$ in the statement of this proposition is non-zero, implying that $x$ is not indecomposable, a contradiction. 

\Case{$k = 2$} Using the same arguments as above, we can see that the only case when $x$ is indecomposable is when $b_1 \in \rayto{\Sigma a}$ and $b_2 \in \exray{a}{\Sigma a}$ (see definitions in \cite[Subsection 2.2]{CSP}). Here, we have $x = x_1$, where $x_1$ is as in \cite[Proposition 4.12]{CSP}. But we have two distinguished triangles $\tri{a}{e_1^{\prime \prime}}{b_1}$ and $\tri{e^{\prime \prime}_1}{x_1}{b_2}$. Therefore, $x$ is the middle term of an extension whose outer-terms are indecomposable. We can then apply the same argument as in the case when $b$ is indecomposable. 
\end{proof}

\begin{lemma}\label{lem:isoextn}
Let $\sX$ be a subcategory of $\Tw$, $\tX$ the corresponding collection of admissible arcs. Suppose there is an outer-isolated vertex $v$ in $\tX$, and write $\tX_1$ (resp. $\tX_2$) to be the subset of $\tX$ whose arcs lie on the right (resp. left) of $v$.
\begin{compactenum}
\item $v$ is an isolated vertex in $\extn{\sX}$. 
\item $\extn{\sX} = \extn{\sX_1} \cup \extn{\sX_2}$, i.e. there are no arcs of the form $(t,u)$ with $u < v < t$. 
\end{compactenum}
\end{lemma}
\begin{proof}
The first statement follows immediately from Theorem \ref{prop:computationofextensions}. Recall that $\extn{\sX} = \bigcup\limits_{n \geq 0} (\sX)_n^\oplus$. The second statement follows easily by induction on $n$ in the same spirit as in the proof of Lemma \ref{lem:arcssms}. 
\end{proof}

\begin{proof}[Proof of Theorem \ref{thm:functfinite}]
$(\Rightarrow)$ Suppose there is an outer-arc in $\tX$. This implies that there are infinitely many outer-arcs $\{\arc{a}_i\}_{i \in \bZ}$ in $\tX$, and since there are only finitely many outer-isolated vertices, there is a $j \in \mathbb{Z}$ for which $t(\arc{a}_{k+1}) = s(\arc{a}_k) +1$, for $k \geq j$, and $t(\arc{a}_j)-1$ is an outer-isolated vertex. Note that, for each $k \geq j$, $t(\arc{a}_k)$ is a right fountain in $\extn{\sX}$. Moreover, since $t(\mathtt{a}_j) -1$ is an outer-isolated vertex, it follows from Lemma \ref{lem:isoextn} that there are no arcs in $\extn{\tX}$ of the form $(t(\arc{a}_k),x)$, with $x < t(\arc{a}_j)-1$. In particular, $t(\arc{a}_k)$ is not a left fountain. Therefore, $\extn{\sX}$ is not covariantly finite, by Theorem \ref{cor:contravariant-finiteness}, and so it is not functorially finite. One can similarly show that $\extn{\sX}$ is not contravariantly finite either. 

$(\Leftarrow)$ Suppose there are no outer-arcs in $\tX$. In particular, this means that there are no outer-isolated vertices in $\tX$. By Theorem \ref{cor:contravariant-finiteness}, we need to show that, given $v \in \mathbb{Z}$, $v$ is a right fountain in $\extn{\tX}$ if and only if $v$ is a left fountain in $\extn{\tX}$. 

If $v$ is an isolated vertex in $\tX$, then it is also an isolated vertex in $\extn{\tX}$ by Lemma \ref{lem:isoextn}, and there is nothing to prove. So suppose $v$ is not isolated in $\tX$. Assume, without loss of generality, that $v$ is the source of an arc $\arc{a}_0$ in $\tX$. By hypothesis, there is an infinite family of arcs $\{\mathtt{a}_i\}_{i \in \mathbb{N}_0}$ in $\tX$ such that $\mathtt{a}_{i+1}$ is the smallest overarc of $\mathtt{a}_i$. We have three cases:

\Case{1} $\tX$ satisfies the following condition:

For each even $i$ (resp. odd $i$), there is a sequence $\{\arc{a}_j^i\}_{j=1}^{k_i}$ of arcs in $\tX$ for which $\arc{a}_i$ is the smallest overarc and such that $t(\mathtt{a}_{j+1}^i) = s(\mathtt{a}^i_j) +1$, $s(\mathtt{a}^i_{k_i}) = s(\mathtt{a}_i) -1$ (resp. $t(\mathtt{a}_1^i) = t(\mathtt{a}_i) +1$), and $\mathtt{a}_{1}^i = \mathtt{a}_{i-1}$ (resp. $\mathtt{a}_{k_i}^i = \mathtt{a}_{i-1}$). 

In this case, it is easy to check that $v$ is both a left and right fountain in $\extn{\sX}$.

\Case{2}  There is $i \in \mathbb{N}$ for which $\mathtt{a}_i$ has at least one inner-isolated vertex smaller than $t(\mathtt{a}_{i-1})$ and at least one inner-isolated vertex bigger than $s(\mathtt{a}_{i-1})$.   
 
Let $z$ be the largest inner-isolated vertex of $\arc{a}_i$ which is smaller than $t(\arc{a}_{i-1})$ and $z^\prime$ be the smallest inner-isolated vertex of $\arc{a}_i$ which is bigger than $s(\arc{a}_{i-1})$. Then it follows from Lemma \ref{lem:isoextn} that the endpoints of any arc in $\extn{\tX}$ incident with $v$ lie in $]z,z^\prime[$. Hence, $v$ is neither a left or right fountain.

\Case{3}  There is an odd $i$ (resp. even $i$) and a sequence $\{\mathtt{a}_j^i\}_{j=1}^{k_i}$ of arcs in $\tX$ for which $\mathtt{a}_i$ is the smallest overarc and such that $t(\mathtt{a}_{j+1}^i) = s(\mathtt{a}^i_j) +1$, $s(\mathtt{a}_{k_i}^i) = s(\mathtt{a}_i) -1$ (resp. $t(\mathtt{a}^i_{1}) = t(\mathtt{a}_i) +1$), and  $\mathtt{a}_1^i = \mathtt{a}_{i-1}$ (resp. $\mathtt{a}_{k_i}^i = \mathtt{a}_{i-1}$).

Let $z$ be the largest inner-isolated vertex of $\arc{a}_i$, which by assumption is smaller than $t(\arc{a}_{i-1})$. We claim that for each $n \in \bN$, we have the following:

$(i)_n$: There are no arcs of the form $(s(\arc{a}_j^i),x)$ in $(\tX)_n$, where $x < z$ and $j= 1, \ldots, k_i$;

$(ii)_n$: There are no arcs in $(\tX)_n$ whose target is $s(\arc{a}_j^i)$, with $j= 1, \ldots, k_i$.

This claim implies that $v$ is neither a left or right fountain. We prove the claim by induction on $n$. Clearly $(i)_1$ and $(ii)_1$ hold, since $(\sX)_1 = \sX$. Suppose $(i)_n$ and $(ii)_n$ hold. Let $\arc{e}$ be an arc in $(\sX)_{n+1} = (\sX)_n * \sX$ incident with $s(\arc{a}_j^i)$, for some $j$, which does not satisfy $(i)_{n+1}$ nor $(ii)_{n+1}$. 

By Lemma \ref{cor:extensionclosurePtolemyII} and Proposition \ref{prop:extensions}, we have $\tri{x}{e}{y}$, with $x \in (\sX)_n$, $y \in \sX$, and $\arc{x}$ and $\arc{y}$ neighbouring arcs satisfying one of the following three cases. 

\Case{$s(\arc{y}) = t(\arc{x})-1$} Suppose $s(\arc{e}) = s(\arc{a}_j^i)$. Since $z$ is isolated, we either have $t(\arc{x}) < z < s(\arc{x})$ or $t(\arc{y}) < z < s (\arc{y})$. However this is not possible, since the first contradicts $(i)_n$ and the latter contradicts $(i)_1$. 

Suppose now that $t(\arc{e}) = s(\arc{a}_j^i)$. Then $t(\arc{y}) = s(\arc{a}_j^i)$, contradicting $(ii)_1$.  

\Case{$s(\arc{x}) = s(\arc{y})+1$} If $s(\arc{a}_j^i) = t(\arc{e}) = t(\arc{x})$, then $\arc{x}$ contradicts $(ii)_n$. On the other hand, if $s(\arc{a}_j^i) = s(\arc{e})$, then $\arc{y}$ would contradict $(ii)_1$. 

\Case{$t(\arc{y}) = t(\arc{x})-1$} If $s(\arc{a}_j^i) = s(\arc{e}) = s(\arc{y})$, and $ t(\arc{y}) < t(\arc{e}) < z < s(\arc{y})$, then $\arc{y}$ would contradict $(i)_1$. Then we have $s(\arc{a}_j^i) = t(\arc{e}) = s(\arc{x})$. By $(i)_n$ and $z$ is isolated, we must have $z < t(\arc{y})$. Hence, since there are no crossings in $\tX$, $\arc{y}$, which is an overarc of $s(\arc{a}_j^i)$, must be an innerarc of $\arc{a}_i$. This contradicts the fact that $\arc{a}_i$ is the smallest overarc of $s(\arc{a}_j^i$, finishing the proof of the claim and of the theorem. 
\end{proof}

\begin{corollary}\label{cor:functoriallyfinite}
Given a $w$-Hom configuration $\sS$ in $\Tw$, $\extn{\sS}$ is functorially finite if and only if $\sS$ is a $w$-simple-minded system.  
\end{corollary}

\section{Mutations via approximations}

In \cite{Dugas12} a mutation of simple-minded systems was introduced. In the last section of this paper, we give a representation theoretic description of mutations of $w$-simple-minded systems, via an iterative procedure of mutations of simple-minded systems. When $w \leq -2$, this mutation procedure is very reminiscent to that of $(|w|-1)$-cluster-tilting objects.

A map $f\colon a \to t$ is {\it right minimal} if for every $g\colon a \to a$ such that $fg=f$, we have that $g$ is an isomorphism. The map $f\colon a \to t$ is a {\it minimal right $\sA$-approximation} if it is both right minimal and a right $\sA$-approximation. There are dual notions of {\it left minimal} and {\it minimal left $\sA$-approximation}.

\begin{lemma}\label{lem:fact1}
Let $\sS$ be a $w$-Hom configuration in $\Tw$ and $s \in \sS$. Suppose $\Sigma^{-1} s$ admits a minimal left $\extn{\sS \setminus \{s\}}$-approximation $\Sigma^{-1} s \rightarrow x_s$. Then each summand $x^\prime$ of $x_s$ must satisfy one of the following conditions:
\begin{compactenum}[(i)]
\item $s(\mathtt{x}^\prime) = s (\Sigma^{-1} \mathtt{s})$ and $t(\mathtt{x}^\prime) < t(\mathtt{s})$,
\item $t(\mathtt{x}^\prime) = s(\Sigma^{-1} \mathtt{s})$, or
\item $t(\mathtt{x}^\prime) = t(\Sigma^{-1} \mathtt{s})$ and $t(\Sigma^{-1} \mathtt{s}) < s(\mathtt{x}^\prime) < s(\mathtt{s})$.
\end{compactenum}
\end{lemma}
\begin{proof}
It is a consequence of Lemma \ref{lem:isoextn} and \cite[Lemma 2.3]{CS3}.
\end{proof}

\begin{proposition}\label{prop:approximationmutation}
Let $\sS$ be a $w$-Hom configuration in $\Tw$ and $s \in \sS$ be such that $\mathtt{s}$ has an overarc in $\tS$. The following conditions hold:
\begin{compactenum}
\item $\Sigma^{-1} s$ has a left $\langle \sS \setminus \{s\} \rangle$-approximation;
\item The cocone $s^*$ of the minimal left $\langle \sS \setminus \{s\} \rangle$-approximation $f$ is indecomposable and it has an overarc in $\tS$;
\item $s^* = s$ if and only if $w = -1$;
\item $\sS \setminus \{s\} \cup \{s^* \}$ is a $w$-Hom-configuration in $\Tw$. 
\end{compactenum}
\end{proposition}

\begin{proof}
In order to prove (1) and show the properties of the cocone, we shall explicitly compute the minimal left $\extn{\sS \setminus \{s\}}$-approximation of $\Sigma^{-1} s$. 

Let $\arc{a}$ be the smallest overarc of $\arc{s}$ in $\tS$. Denote by $\{\arc{b}_i\}_{i=1}^k$ (resp. $\{\arc{c}_i\}_{i=1}^{k^\prime}$) the set of innerarcs $\arc{a}$ (resp. $\arc{s}$) in $\tS$. Assume these arcs are ordered from left to right, and write $\arc{s} = \arc{b}_i$, for some $i \in \{1, \ldots, k\}$. 
 
We now define three arcs $\arc{e}_1$, $\arc{e}_2$ and $\arc{s}^\prime$, with a case by case analysis in terms of the inner-isolated vertices of $\arc{a}$ and $\arc{s}$. The element $e_1 \oplus e_2$ will define the minimal left $\extn{\sS \setminus \{s\}}$-approximation of $\Sigma^{-1} s$, and $\arc{s}^\prime$ will be its cone. 

\Case{1} $\arc{a}$ has no inner-isolated vertices in $\tS$.

We must have $w=-1$ and so $\arc{s}$ has no inner-isolated vertices either. Take
\[
\arc{e}_1:= (s(\Sigma^{-1} \arc{s}),t(\arc{s})-1) \quad
\arc{e}_2:= \begin{cases}
(s(\arc{s})-1, t(\arc{s})+1) &\mbox{ if } l(\arc{s}) > 1 \\
0 &\mbox{ if } l(\arc{s}) = 1,
\end{cases}
\] 
and 
\[
\arc{s}^\prime := \begin{cases}
(s(\arc{e}_2),t(\arc{e}_1)) &\mbox{ if } e_2 \neq 0 \\
(t(\arc{s}),t(\arc{e_1})) &\mbox{ if } e_2 = 0.
\end{cases}
\]

\Case{2} $\mathtt{a}$ has an inner-isolated vertex in $\tS$ larger than $s(\arc{s})$. 

We must have $w\neq -1$ and so $\arc{s}$ also has an inner-isolated vertex. Let $v$ (resp. $v^\prime$) be the smallest inner-isolated vertex in $]s(\arc{s}),s(\arc{a})[$ (resp. $]t(\arc{s}),s(\arc{s})[$). Take
\[
\arc{e}_1 := \begin{cases}
(s(\arc{b}_j), s(\Sigma^{-1} \arc{s})) &\mbox{ if } v= s(\arc{b}_j)+1, j \geq i+1 \\
0 &\mbox{ if } v= s(\arc{b}_i)+1,
\end{cases}
\]
\[
\arc{e}_2:= \begin{cases}
(s(\arc{c}_i), t(\Sigma^{-1} \arc{s})) &\mbox{ if } v^\prime = s(\arc{c}_i)+1 \\
0 &\mbox{ if } v^\prime = t(\arc{s})+1, 
\end{cases}
\]
and 
\[
\arc{s}^\prime := \begin{cases}
(s(\arc{e}_1),s(\arc{e_2})) &\mbox{ if } e_1 \neq 0, e_2 \neq 0 \\
(s(\arc{s}),s(\arc{e_2})) &\mbox{ if } e_1 = 0, e_2 \neq 0 \\
(s(\arc{e}_1),t(\arc{s})) &\mbox{ if } e_1 \neq 0, e_2 = 0 \\
\arc{s} &\mbox{ if } e_1 = e_2 = 0. 
\end{cases}
\]

\Case{3} $\arc{a}$ has an inner-isolated vertex in $\tS$ which is smaller than $t(\mathtt{s})$ and no inner-isolated vertex larger than $s(\arc{s})$. 

We must have $w \neq -1$ and so $\arc{s}$ has an inner-isolated vertex. Let $v$ (resp. $v^\prime$) the smallest inner-isolated vertex of $\arc{a}$ (resp. $\arc{s}$). Take $\arc{e}_2$ to be as in Case 2 and 
\[
\arc{e}_1:= \begin{cases}
(s(\Sigma^{-1} \arc{s}),s(\arc{b}_j)) &\mbox{ if } v= s(\arc{b}_j)+1 \\
(s(\Sigma^{-1} \arc{s}),t(\arc{a})) &\mbox{ if } v= t(\arc{a})+1,
\end{cases} \quad
\arc{s}^\prime := \begin{cases}
(s(\arc{e}_2),t(\arc{e}_1)) &\mbox{ if } e_2 \neq 0 \\
(t(\arc{s}), t(\arc{e}_1)) &\mbox{ if } e_2 = 0.
\end{cases}
\]

Note that $\Hom (\Sigma^{-1} s, e_1)$ and $\Hom (\Sigma^{-1} s, e_2)$ are one-dimensional spaces. We claim that the map $\xymatrix{\Sigma^{-1} s \ar@{->}[r]^{\rowmat{f_1}{f_2}^t} & e_1 \oplus e_2}$, with $f_1, f_2 \neq 0$ is the minimal left $\langle \sS \setminus \{s\} \rangle$-approximation of $\Sigma^{-1} s$. Indeed, it is clear that $e_1 \oplus e_2 \in \langle \sS \setminus \{s\} \rangle$. Given a non-zero map $g: \Sigma^{-1} s \rightarrow x_s$, with $x_s \in \langle \sS \setminus \{s\} \rangle$, it follows from Lemma \ref{lem:fact1} and the factorisation properties of $\Tw$ (cf. \cite[Proposition 2.3]{CSP}, see also \cite[Propositions 2.1 and 2.2]{HJ2}) that $g$ factors through $\rowmat{f_1}{f_2}^t$. The minimality follows from the fact that $\{e_1, e_2\}$ is an orthogonal set. 

Now, we can deduce from Proposition \ref{prop:extensions} that we have a triangle of the form $$\xymatrix{\Sigma^{-1} s \ar@{->}[r]^{\rowmat{f_1}{f_2}^t} & e_1 \oplus e_2 \ar@{->}[r] & s^\prime \ar@{->}[r] & s},$$ where $\arc{s}^\prime$ is defined above. Hence, the cocone $s^*\coloneqq \Sigma^{-1} s^\prime$ of $\rowmat{f_1}{f_2}^t$ is indecomposable, and the corresponding arc is given by: $\arc{s}^* = \arc{s}$ in Case 1 (i.e. when $w = -1$), $\arc{s}^* = (v,v^\prime) \neq \arc{s}$ in Case 2, and $\arc{s}^* = (v^\prime,v) \neq \arc{s}$ in Case 3. Given that $v, v^\prime$ are inner-isolated vertices of $\arc{a}$, $\arc{s}$ and $\arc{s}^\prime$ have the same smallest overarc. 

Since $\arc{s}^*$ is incident with isolated vertices in $\tS \setminus \{s\}$ and it does not cross any of its arcs, we have that $\sS \setminus \{s\} \cup \{s^*\}$ is $w$-orthogonal. It follows from Proposition \ref{prop:mutationsT-m} that $\sS \setminus \{s\} \cup \{s^*\}$ is in fact a $w$-Hom-configuration.  
\end{proof}

The object $s^*$ is said to be the \emph{left mutation of $\sS$ at $s$}. Note that there is a dual version of Proposition \ref{prop:approximationmutation}, where one considers the cone of the minimal right $\extn{\sS \setminus \{s\}}$-approximation of $\Sigma s$. This cone is said to be the \emph{right mutation of $\sS$ at $s$}. 

Recall that if $s$ is an object in a $w$-Hom configuration $\sS$ in $\Tw$ such that the corresponding arc is not an outer-arc, then $\sS$ can be mutated at $s$ in precisely $|w|$ ways. We will now show that, by iteratively performing the left mutation $|w|$ times, we get all the completions of $\sS \setminus \{s\}$. Note that this iterative procedure is possible, since $s^*$ has an overarc in $\sS \setminus \{s\}$. 
 
\begin{notation} 
Let $\sS$ be a $w$-simple-minded system in $\Tw$, and $s_0 \in \sS$. We denote by 
\begin{equation}\label{eq:iterativeleftmutation}
\trilabels{s_i}{\Sigma^{-1} s_{i-1}}{a_{i-1}}{\sigma_{i-1}}{f_{i-1}}{g_{i-1}},
\end{equation}
 the completion to a triangle of the minimal left $\extn{\sS \setminus \{s_0\}}$-approximation $f_{i-1}$ of $\Sigma^{-1} s_{i-1}$, for each $i = 1, \ldots, |w|$. 
 \end{notation}

\begin{lemma}\label{lem:sigma0non-zero}
The map $\sigma_0$ is non-zero.
\end{lemma}
\begin{proof}
Let $e_1$ and $e_2$ be as in proof of Proposition \ref{prop:approximationmutation}. We can see, by checking all the cases in the proof, that $e_1, e_2 \neq \Sigma^{-1} s_0$. If $\sigma_0 = 0$, we would have $e_1 \oplus e_2 \simeq \Sigma^{-1} s_0 \oplus \Sigma s_1$. In particular, $\Sigma^{-1} s_0$ would coincide with one of the $e_i$s, a contradiction.
\end{proof}

\begin{lemma}\label{prop:differentsis}
Let $w \leq -2$. The composition $\Sigma^{w+2} \sigma_0 \Sigma^{w+3} \sigma_1 \cdots \Sigma^{-1} \sigma_{|w|-3} \sigma_{|w|-2}$ is non-zero. In particular, $\Hom (s_j, \Sigma^{i-j} s_i) \neq 0$ and $s_i \not\simeq s_j$, for $0 \leq i < j \leq -w-1$.  
\end{lemma}

\begin{proof}
Suppose $\Sigma^{w+2} \sigma_0 \Sigma^{w+3} \sigma_1 \cdots \Sigma \sigma_{|w|-3} \sigma_{|w|-2} = 0$. Then there is a map $\alpha \colon a_{|w|-2} \rightarrow \Sigma^{w+1} s_0$ such that $\Sigma^{w+2} \sigma_0 \Sigma^{w+3} \sigma_1 \cdots \Sigma^{-1} \sigma_{|w|-3} = \alpha f_{|w|-2}$. But $a_{|w|-2} \in \langle \sS \setminus \{s_0\} \rangle$ and $\Hom (\langle \sS \setminus \{s_0\} \rangle, \Sigma^{w+1} s_0) = 0$, by hypothesis. Hence, $\Sigma^{w+2} \sigma_0 \Sigma^{w+3} \sigma_1 \cdots \Sigma^{-1} \sigma_{|w|-3} = 0$. Proceeding in this manner, we deduce that $\Sigma^{w+2} \sigma_0 \Sigma^{w+3} \sigma_1 \cdots \Sigma^{i} \sigma_{|w|-2+i} = 0$, for $w+2 \leq i \leq -1$. In particular, $\Sigma^{w+2} \sigma_0 = 0$, and so $\sigma_0 = 0$, contradicting Lemma \ref{lem:sigma0non-zero}. Therefore $\Sigma^{w+2} \sigma_0 \Sigma^{w+3} \sigma_1 \cdots \Sigma^{-1} \sigma_{|w|-3} \sigma_{|w|-2} \neq 0$. This implies that $\Sigma^{w+i+1} \sigma_{i-1} \Sigma^{w+i} \sigma_{i-2} \cdots \Sigma^{w+j+1} \sigma_{j-1}$ is also non-zero, and so $\Hom (s_j, \Sigma^{i-j} s_i) \neq 0$, for $0 \leq i < j \leq -w-1$.

Now, suppose, $s_i \simeq s_j$, for some $0 \leq i < j \leq -w-1$. On one hand, $\Ext^k (s_j,s_i) \neq 0$, for each $w+1 \leq k \leq -1$, by the first part of the proof. On the other hand, since $\sS \setminus \{s_0\} \cup \{s_i\}$ is a $w$-Hom-configuration, we have $\Ext^k (s_j,s_i) \simeq \Ext^k (s_i,s_i) = 0$, for each $w+1 \leq k \leq -1$, a contradiction.
\end{proof}

The following is an immediate consequence of Propositions \ref{prop:mutationsT-m}, \ref{prop:approximationmutation} and Lemma \ref{prop:differentsis}.  

\begin{theorem}
Let $\sS$ be a $w$-Hom-configuration in $\Tw$ and suppose $s_0 \in \sS$ has an overarc in $\sS$. Then the objects $s_0, s_1, \ldots, s_{|w|-1}$ as in \eqref{eq:iterativeleftmutation} are all the completions of $\sS \setminus \{s_0\}$.  
\end{theorem}

In particular, every mutation of a $w$-simple-minded system can be described in terms of left (or right) mutations. The same does not hold for any of the $w$-Hom configurations which are not $w$-simple-minded systems. The following example illustrates the problem in these cases.    

\begin{example}
Let $w= -3$ and $\sS$ be the $w$-Riedtmann configuration with precisely two outer-isolated vertices, say $0$ and $13$, and formed by arcs of minimum length. Note that every arc in $\sS$ is then an outer-arc. Consider an arc which lies between the two outer-isolated vertices, say $\arc{s} \coloneqq (8,5)$. Then, the left mutation of $\arc{s}$ is $\arc{s}^\prime \coloneqq (13,6)$. We know that $(7,0)$ is the other possible replacement of $s$. However, we cannot obtain this object by left mutating $\arc{s}^\prime$, since this object does not lie between a pair of outer-isolated vertices, and therefore its negative shift does not admit a left $\extn{\sS \setminus \{s\}}$-approximation.    
\end{example}

\subsection*{Acknowledgments.}
The author would like to thank Funda\c{c}\~ao para a Ci\^encia e Tecnologia, for their financial support through Grant SFRH/BPD/90538/2012.


\begin{thebibliography}{99}
\bibitem{Aihara-Iyama}
Aihara, T., Iyama, O., 
Silting mutation in triangulated categories, \emph{J. Lond. Math. Soc.} {\bf 85} (2012), no. 2, 633--668. 
\bibitem{Bridgeland}
Bridgeland, T.,
Stability conditions on triangulated categories, \emph{Ann. of Math.} \textbf{166} (2007), 317--345.
\bibitem{BMRRT}
Buan, A.B., Marsh, R.J., Reineke,M., Reiten, I., Todorov, G., 
Tilting theory and cluster combinatorics, \emph{Adv. Math.} {\bf 204} (2006), 572--618. 
\bibitem{BRT}
Buan, A.B., Reiten, I., Thomas, H.,
Three kinds of mutations, \emph{J. Algebra} \textbf{339}, no.1 (2011), 97--113.
\bibitem{CS1}
Coelho Sim\~oes, R., 
Hom-configurations and noncrossing partitions, \emph{J. Algebraic. Combin.} \textbf{35} (2012), no. 2, 311--343. 
\bibitem{CS2}
Coelho Sim\~oes, R., 
Maximal rigid objects as noncrossing bipartite graphs, \emph{Algebr. Represent. Theory} \textbf{16} (2013), 1243--1272. 
\bibitem{CS3}
Coelho Sim\~oes, R.,
Hom-configurations in triangulated categories generated by spherical objects, \emph{J. Pure Appl. Algebra} \textbf{219} (2015), 3322--3336.
\bibitem{CSP}
Coelho Sim\~oes, R., Pauksztello, D.,
Torsion pairs in a triangulated category generated by a spherical object, \emph{J. Algebra} \textbf{448} (2016), 1--47. 
\bibitem{Dugas12}
Dugas, A., 
Torsion pairs and simple-minded systems in triangulated categories, \emph{Appl. Categ. Structures} \textbf{23} (2015), no. 3, 507--526.
\bibitem{FY}
Fu, C., Yang, D., 
The Ringel-Hall Lie algebra of a spherical object, preprint (2011), math.RT/1103.1241v2.
\bibitem{HJ1}
Holm, T., J\o rgensen, P., 
On a cluster category of infinite Dynkin type, and the relation to triangulations of the infinity-gon, \emph{Math. Z.} \textbf{270} (2012), no. 1-2, 277--295.
\bibitem{HJ2}
Holm, T., J\o rgensen, P.,
Cluster-tilting vs. weak cluster-tilting in Dykin type $A$ infinity, \emph{Forum Math.} \textbf{27} (2015), 1117--1137. 
\bibitem{HJY}
Holm, T., J\o rgensen, P., Yang, D., 
Sparseness of t-structures and negative Calabi-Yau dimension in triangulated categories generated by a spherical object, \emph{Bull. London Math. Soc.} \textbf{45} (2013), 120--130. 
\bibitem{J}
J\o rgensen, P.,
Auslander-Reiten theory over topological spaces, \emph{Comment. Math. Helv.} \textbf{79} (2004), 160--182.
\bibitem{KYZ}
Keller, B., Yang, D., Zhou, G.,
The Hall algebra of a spherical object, \emph{J. London Math. Soc.} (2) \textbf{80} (2009), 771--784.
\bibitem{King-Qiu}
King, A., Qiu, Y.,
Exchange graphs of acyclic Calabi-Yau categories, \emph{Adv. Math.} {\bf 285} (2015), 1106--1154. 
\bibitem{KL}
Koenig, S., Liu, Y.,
Simple-minded systems in stable module categories, \emph{Q. J. Math.} {\bf 63} (2012), no. 3, 653--674.
\bibitem{KY}
Koenig, S., Yang, D., 
Silting objects, simple-minded collections, t-structures, co-t-structures for finite-dimensional algebras, \emph{Doc. Math.} {\bf 19} (2014), 403--438.
\bibitem{Ng}
Ng, P.,
A characterization of torsion theories in the cluster category of type $A_\infty$, preprint (2010), math.RT/1005.4364v1.
\bibitem{Po}
Pogorza\l y, Z., Algebras stably equivalent to self-injective special biserial algebras, \emph{Comm. in Algebra} {\bf 22} (1994), no. 4, 1127--1160. 
\bibitem{Riedtmann}
Riedtmann, C.; Representation-finite selfinjective algebras of class $A_n$, Representation theory II (Prof. Second Internat. Conf., Carleton Univ., Ottawa, Ont. 1979) 449--520, \emph{Lecture Notes in Math.} \textbf{832}, Springer, Berlin, 1980.
\bibitem{ZZ}
Zhou, Y., Zhu, B., 
Mutation of torsion pairs in triangulated categories and its geometric realization, preprint (2011), math.RT/1105.3521v1.
\bibitem{Zhu08}
Zhu, B.,
Generalized cluster complexes via quiver representations, \emph{J. Algebraic. Combin.} \textbf{27} (2008), no. 1, 35--54.
\end{thebibliography}
\end{document}